\documentclass[12pt]{amsart}       
\usepackage{txfonts}
\usepackage{amssymb}
\usepackage{eucal}
\usepackage{graphicx}
\usepackage{amsmath}
\usepackage{amscd}
\usepackage[all]{xy}           
\usepackage{amsfonts,latexsym}
\usepackage{xspace}
\usepackage{epsfig}
\usepackage{float}
\usepackage{color}
\usepackage{fancybox}
\usepackage{colordvi}
\usepackage{multicol}
\usepackage{colordvi}
\usepackage[colorlinks,final,backref=page,hyperindex,hypertex]{hyperref}
\usepackage[active]{srcltx} 
\usepackage{tikz}

\usepackage{graphicx}
\usepackage{mathrsfs}
\newtheorem{theorem}{Theorem}[section]
\newtheorem{prop}[theorem]{Proposition}

\newtheorem{prop-def}{Proposition-Definition}[section]
\newtheorem{coro-def}{Corollary-Definition}[section]

\theoremstyle{definition}
\newtheorem{defn}[theorem]{Definition}
\newtheorem{remark}[theorem]{Remark}

\newtheorem{exam}[theorem]{Example}

\newcommand{\nc}{\newcommand}
\nc{\tred}[1]{\textcolor{red}{#1}}
\nc{\tblue}[1]{\textcolor{blue}{#1}}
\nc{\tgreen}[1]{\textcolor{green}{#1}}
\nc{\tpurple}[1]{\textcolor{purple}{#1}}
\nc{\btred}[1]{\textcolor{red}{\bf #1}}
\nc{\btblue}[1]{\textcolor{blue}{\bf #1}}
\nc{\btgreen}[1]{\textcolor{green}{\bf #1}}
\nc{\btpurple}[1]{\textcolor{purple}{\bf #1}}
\nc{\NN}{{\mathbb N}}
\nc{\ncsha}{{\mbox{\cyr X}^{\mathrm NC}}} \nc{\ncshao}{{\mbox{\cyr
X}^{\mathrm NC}_0}}

\newcommand{\efootnote}[1]{}

\renewcommand{\textbf}[1]{}

\newcommand{\delete}[1]{}

\nc{\mlabel}[1]{\label{#1}}  
\nc{\mcite}[1]{\cite{#1}}  
\nc{\mref}[1]{\ref{#1}}  
\nc{\mbibitem}[1]{\bibitem{#1}} 

\delete{
\nc{\mlabel}[1]{\label{#1}{\hfill \hspace{1cm}{\bf{{\ }\hfill(#1)}}}}
\nc{\mcite}[1]{\cite{#1}{{\bf{{\ }(#1)}}}}  
\nc{\mref}[1]{\ref{#1}{{\bf{{\ }(#1)}}}}  
\nc{\mbibitem}[1]{\bibitem[\bf #1]{#1}} 
}

\nc{\opa}{\ast} \nc{\opb}{\odot} \nc{\op}{\bullet} \nc{\pa}{\frakL}
\nc{\arr}{\rightarrow} \nc{\lu}[1]{(#1)} \nc{\mult}{\mrm{mult}}
\nc{\diff}{\mathfrak{Diff}}
\nc{\opc}{\sharp}\nc{\opd}{\natural}
\nc{\ope}{\circ}
\nc{\dpt}{\mathrm{d}}
\nc{\hck}{H_{RT}}
\nc{\vdf}{\calf}
\nc{\ldf}{\calf_\ell}
\nc{\hlf}{H_\ell}
\nc{\onek}{\mathbf{1}_\bfk}

\nc{\tforall}{\quad \text{ for all }}

\nc{\AW}{\mathcal{A}}

\nc{\rba}{Rota-Baxter algebra\xspace}
\nc{\rbas}{Rota-Baxter algebras\xspace}

\nc{\ari}{\mathrm{ar}}

\nc{\lef}{\mathrm{lef}}

\nc{\Sh}{\mathrm{ST}}

\nc{\Cr}{\mathrm{Cr}}

\nc{\st}{{Schr\"oder tree}\xspace}
\nc{\sts}{{Schr\"oder trees}\xspace}

\nc{\vertset}{\Omega} 

\nc{\assop}{\quad \begin{picture}(5,5)(0,0)
\line(-1,1){10}
\put(-2.2,-2.2){$\bullet$}
\line(0,-1){10}\line(1,1){10}
\end{picture} \quad \smallskip}

\nc{\operator}{\begin{picture}(5,5)(0,0)
\line(0,-1){6}
\put(-2.6,-1.8){$\bullet$}
\line(0,1){9}
\end{picture}}

\nc{\idx}{\begin{picture}(6,6)(-3,-3)
\put(0,0){\line(0,1){6}}
\put(0,0){\line(0,-1){6}}
\end{picture}}

\nc{\pb}{{\mathrm{pb}}}
\nc{\Lf}{{\mathrm{Lf}}}

\nc{\lft}{{left tree}\xspace}
\nc{\lfts}{{left trees}\xspace}

\nc{\fat}{{fundamental averaging tree}\xspace}

\nc{\fats}{{fundamental averaging trees}\xspace}
\nc{\avt}{\mathrm{Avt}}

\nc{\rass}{{\mathit{RAss}}}

\nc{\aass}{{\mathit{AAss}}}

\nc{\vin}{{\mathrm Vin}}    
\nc{\lin}{{\mathrm Lin}}    
\nc{\inv}{\mathrm{I}n}
\nc{\gensp}{V} 
\nc{\genbas}{\mathcal{V}} 
\nc{\bvp}{V_P}     
\nc{\gop}{{\,\omega\,}}     

\nc{\bin}[2]{ (_{\stackrel{\scs{#1}}{\scs{#2}}})}  
\nc{\binc}[2]{ \left (\!\! \begin{array}{c} \scs{#1}\\
    \scs{#2} \end{array}\!\! \right )}  
\nc{\bincc}[2]{  \left ( {\scs{#1} \atop
    \vspace{-1cm}\scs{#2}} \right )}  
\nc{\bs}{\bar{S}} \nc{\cosum}{\sqsubset} \nc{\la}{\longrightarrow}
\nc{\rar}{\rightarrow} \nc{\dar}{\downarrow} \nc{\dprod}{**}
\nc{\dap}[1]{\downarrow \rlap{$\scriptstyle{#1}$}}
\nc{\md}{\mathrm{dth}} \nc{\uap}[1]{\uparrow
\rlap{$\scriptstyle{#1}$}} \nc{\defeq}{\stackrel{\rm def}{=}}
\nc{\disp}[1]{\displaystyle{#1}} \nc{\dotcup}{\
\displaystyle{\bigcup^\bullet}\ } \nc{\gzeta}{\bar{\zeta}}
\nc{\hcm}{\ \hat{,}\ } \nc{\hts}{\hat{\otimes}}
\nc{\barot}{{\otimes}} \nc{\free}[1]{\bar{#1}}
\nc{\uni}[1]{\tilde{#1}} \nc{\hcirc}{\hat{\circ}} \nc{\lleft}{[}
\nc{\lright}{]} \nc{\lc}{\lfloor} \nc{\rc}{\rfloor}
\nc{\curlyl}{\left \{ \begin{array}{c} {} \\ {} \end{array}
    \right .  \!\!\!\!\!\!\!}
\nc{\curlyr}{ \!\!\!\!\!\!\!
    \left . \begin{array}{c} {} \\ {} \end{array}
    \right \} }
\nc{\longmid}{\left | \begin{array}{c} {} \\ {} \end{array}
    \right . \!\!\!\!\!\!\!}
\nc{\onetree}{\bullet} \nc{\ora}[1]{\stackrel{#1}{\rar}}
\nc{\ola}[1]{\stackrel{#1}{\la}}
\nc{\ot}{\otimes} \nc{\mot}{{{\boxtimes\,}}}
\nc{\otm}{\overline{\boxtimes}} \nc{\sprod}{\bullet}
\nc{\scs}[1]{\scriptstyle{#1}} \nc{\mrm}[1]{{\rm #1}}
\nc{\margin}[1]{\marginpar{\rm #1}}   
\nc{\dirlim}{\displaystyle{\lim_{\longrightarrow}}\,}
\nc{\invlim}{\displaystyle{\lim_{\longleftarrow}}\,}
\nc{\mvp}{\vspace{0.3cm}} \nc{\tk}{^{(k)}} \nc{\tp}{^\prime}
\nc{\ttp}{^{\prime\prime}} \nc{\svp}{\vspace{2cm}}
\nc{\vp}{\vspace{8cm}} \nc{\proofbegin}{\noindent{\bf Proof: }}
\nc{\proofend}{$\blacksquare$ \vspace{0.3cm}}
\nc{\modg}[1]{\!<\!\!{#1}\!\!>}
\nc{\intg}[1]{F_C(#1)} \nc{\lmodg}{\!
<\!\!} \nc{\rmodg}{\!\!>\!}
\nc{\cpi}{\widehat{\Pi}}
\nc{\sha}{{\mbox{\cyr X}}}  

\newfont{\scyr}{wncyr10 scaled 550}
\nc{\ssha}{\mbox{\bf \scyr X}}
\nc{\shap}{{\mbox{\cyrs X}}} 
\nc{\shpr}{\diamond}    
\nc{\shp}{\ast} \nc{\shplus}{\shpr^+}
\nc{\shprc}{\shpr_c}    
\nc{\msh}{\ast} \nc{\zprod}{m_0} \nc{\oprod}{m_1}
\nc{\vep}{\epsilon} \nc{\labs}{\mid\!} \nc{\rabs}{\!\mid}
\nc{\sqmon}[1]{\langle #1\rangle}

\nc{\mmbox}[1]{\mbox{\ #1\ }} \nc{\dep}{\mrm{dep}} \nc{\fp}{\mrm{FP}}
\nc{\rchar}{\mrm{char}} \nc{\End}{\mrm{End}} \nc{\Fil}{\mrm{Fil}}
\nc{\Mor}{Mor\xspace} \nc{\gmzvs}{gMZV\xspace}
\nc{\gmzv}{gMZV\xspace} \nc{\mzv}{MZV\xspace}
\nc{\mzvs}{MZVs\xspace} \nc{\Hom}{\mrm{Hom}} \nc{\id}{\mrm{id}}
\nc{\im}{\mrm{im}} \nc{\incl}{\mrm{incl}} \nc{\map}{\mrm{Map}}
\nc{\mchar}{\rm char} \nc{\nz}{\rm NZ} \nc{\supp}{\mathrm Supp}

\nc{\Alg}{\mathbf{Alg}} \nc{\Bax}{\mathbf{Bax}} \nc{\bff}{\mathbf f}
\nc{\bfk}{{\bf k}} \nc{\bfone}{{\bf 1}} \nc{\bfx}{\mathbf x}
\nc{\bfy}{\mathbf y}
\nc{\base}[1]{\bfone^{\otimes ({#1}+1)}} 
\nc{\Cat}{\mathbf{Cat}}

\nc{\detail}{\marginpar{\bf More detail}
    \noindent{\bf Need more detail!}
    \svp}
\nc{\Int}{\mathbf{Int}} \nc{\Mon}{\mathbf{Mon}}
\nc{\rbtm}{{shuffle }} \nc{\rbto}{{Rota-Baxter }}
\nc{\remarks}{\noindent{\bf Remarks: }} \nc{\Rings}{\mathbf{Rings}}
\nc{\Sets}{\mathbf{Sets}} \nc{\wtot}{\widetilde{\odot}}
\nc{\wast}{\widetilde{\ast}} \nc{\bodot}{\bar{\odot}}
\nc{\bast}{\bar{\ast}} \nc{\hodot}[1]{\odot^{#1}}
\nc{\hast}[1]{\ast^{#1}} \nc{\mal}{\mathcal{O}}
\nc{\tet}{\tilde{\ast}} \nc{\teot}{\tilde{\odot}}
\nc{\oex}{\overline{x}} \nc{\oey}{\overline{y}}
\nc{\oez}{\overline{z}} \nc{\oef}{\overline{f}}
\nc{\oea}{\overline{a}} \nc{\oeb}{\overline{b}}
\nc{\weast}[1]{\widetilde{\ast}^{#1}}
\nc{\weodot}[1]{\widetilde{\odot}^{#1}} \nc{\hstar}[1]{\star^{#1}}
\nc{\lae}{\langle} \nc{\rae}{\rangle}
\nc{\lf}{\lfloor}
\nc{\rf}{\rfloor}

\nc{\QQ}{{\mathbb Q}}
\nc{\RR}{{\mathbb R}} \nc{\ZZ}{{\mathbb Z}}

\nc{\cala}{{\mathcal A}} \nc{\calb}{{\mathcal B}}
\nc{\calc}{{\mathcal C}}
\nc{\cald}{{\mathcal D}} \nc{\cale}{{\mathcal E}}
\nc{\calf}{{\mathcal F}} \nc{\calg}{{\mathcal G}}
\nc{\calh}{{\mathcal H}} \nc{\cali}{{\mathcal I}}
\nc{\call}{{\mathcal L}} \nc{\calm}{{\mathcal M}}
\nc{\caln}{{\mathcal N}} \nc{\calo}{{\mathcal O}}
\nc{\calp}{{\mathcal P}} \nc{\calr}{{\mathcal R}}
\nc{\cals}{{\mathcal S}} \nc{\calt}{{\mathcal T}}
\nc{\calu}{{\mathcal U}} \nc{\calw}{{\mathcal W}} \nc{\calk}{{\mathcal K}}
\nc{\calx}{{\mathcal X}} \nc{\CA}{\mathcal{A}}

\nc{\fraka}{{\mathfrak a}} \nc{\frakA}{{\mathfrak A}}
\nc{\frakb}{{\mathfrak b}} \nc{\frakB}{{\mathfrak B}}
\nc{\frakc}{{\mathfrak c}}
\nc{\frakD}{{\mathfrak D}} \nc{\frakF}{\mathfrak{F}}
\nc{\frakf}{{\mathfrak f}} \nc{\frakg}{{\mathfrak g}}
\nc{\frakH}{{\mathfrak H}} \nc{\frakL}{{\mathfrak L}}
\nc{\frakM}{{\mathfrak M}} \nc{\bfrakM}{\overline{\frakM}}
\nc{\frakm}{{\mathfrak m}} \nc{\frakP}{{\mathfrak P}}
\nc{\frakN}{{\mathfrak N}} \nc{\frakp}{{\mathfrak p}}
\nc{\frakS}{{\mathfrak S}} \nc{\frakT}{\mathfrak{T}}
\nc{\frakX}{{\mathfrak X}}
\nc{\BS}{\mathbb{S
}}

\font\cyr=wncyr10 \font\cyrs=wncyr7
\nc{\li}[1]{\textcolor{red}{#1}}
\nc{\lir}[1]{\textcolor{red}{Li:#1}}
\nc{\yi}[1]{\textcolor{blue}{Yi: #1}}
\nc{\shuangjian}[1]{\textcolor{purple}{shuangjian:#1}}
\nc{\revise}[1]{\textcolor{purple}{#1}}

\nc{\ID}{{\rm I}}\nc{\lbar}[1]{\overline{#1}}\nc{\bre}{{\rm bre}}
\nc{\sd}{\cals}\nc{\rb}{\rm RB}\nc{\A}{\rm A}\nc{\LL}{\rm L}\nc{\tx}{\tilde{X}}
\nc{\col}{\Delta_{RT}}\nc{\mul}{m_{RT}}\nc{\ul}{u_{RT}}\nc{\epl}{\varepsilon_{RT}}
\nc{\hl}{H_{RT}}\nc{\arro}[1]{#1}\nc{\px}{P_{\tx}}\nc{\pw}{P_{\mathfrak{w}}}\nc{\pl}{B_\omega^+}
\nc{\pp}{\pl}\nc{\ppp}[1]{B^+(#1)}\nc{\dw}{\diamond_{\mathfrak{w}}}\nc{\dl}{\diamond_{\rm \ell}}
\nc{\ncshaw}{\sha^{{\rm NC}}_{\Omega}}\nc{\ncshal}{\sha^{{\rm NC}}_{{\rm RT}}}
\nc{\ver}{\rm V}\nc{\ld}{l}\nc{\del}{\Delta_{{\rm \ell}}}\nc{\epsl}{\epsilon_{{\rm \ell}}}
\nc{\uul}{u_{{\rm \ell}}}\nc{\oneh}{\mathbf{1}}\nc{\onew}{\mathbf{1}}
\nc{\etree}{1} \nc{\conc}{m_{RT}}
\nc{\hrtb}{H_{RT}(X\sqcup\Omega)} \nc{\hrts}{H_{RT}(X, \Omega)}\nc{\rts}{\mathcal{T}(X, \Omega)}\nc{\rfs}{\mathcal{F}(X, \Omega)} \nc{\ncshall}{\sha^{{\rm NC}}_{{\rm RT}}} \nc{\ldl}{\leq_{\mathrm{dl}}} \nc{\pla}{B_{\alpha}^{+}} \nc{\plb}{B_{\beta}^{+}}

\nc{\bim}[1]{#1}  \nc{\shaop}{\sha_{\Omega}^{+}}  \nc{\shao}{\sha_{\Omega}}
\nc{\bbim}[2]{#1 #2} \nc{\bbbim}[2]{#1,\, #2} \nc{\RBF}{{\rm RBF}}
\nc{\frbf}{F_{\RBF}} \nc{\shaf}{\ssha_{\tiny{\Omega}}} \nc{\sham}{\diamond_{\Omega}}

\nc{\dnx}{\Delta_n A} \nc{\dx}{\Delta A} \nc{\dgp}{{\rm deg_{P}}}
\nc{\dgt}{{\rm deg_{T}}} \nc{\dg}{{\rm deg}} \nc{\ida}{ID($A$)} \nc{\tu}{\tilde{u}} \nc{\tv}{\tilde{v}}
 \nc{\fbase}{\calb} \nc{\LF}{\mathrm{RF}} \nc{\FFA}{\mathrm{LF}} \nc{\irr}{\mathrm{Irr}}
 \nc{\result}{\bfk\mathrm{Irr}(S_n)}  \nc{\I}{I_{\mathrm{ID},n}^0}
 \nc{\nrs}{\calr_n^\star} \nc{\ii}{\mathrm{I}} \nc{\iii}{\mathrm{II}}
\nc{\intl}{{\rm int}}\nc{\ws}[1]{{#1}}\nc{\deleted}[1]{\delete{#1}}\nc{\plas}{placements\xspace}

\nc{\Id}{\mathrm{Id}} \nc{\Irr}{\mathrm{Irr}}

\nc{\tos}{totally ordered set } \nc{\nes}{nonempty set}
\nc{\baa}{\succ_\alpha}
\nc{\laa}{\prec_\alpha}
\nc{\bba}{\succ_\beta}
\nc{\lba}{\prec_\beta}
\nc{\bdt} {\succ_\delta}
\nc{\lia} {\prec_\iota}
\nc{\bka}{\succ_\kappa}
\nc{\lmu}{\prec_\mu}
\nc{\Po}{(P_\omega)_{\omega\in \Omega}}

\usetikzlibrary{calc}
\newlength\xch\newlength\dbj
\newif\ifqdd
\begin{document}

\title[The Cohomology of RCW Reynolds operators  and NS-pre-Lie algebras]{The Cohomology of relative cocycle weighted Reynolds operators and NS-pre-Lie algebras}
%

\author{Shuang-Jian Guo}
\address{School of Mathematics and Statistics, Guizhou University of Finance and Economics, Guiyang, Guizhou 550025, P.\,R. China}
\email{shuangjianguo@126.com}

\author{Yi Zhang$^{*}$}\footnotetext{*Corresponding author.}
\address{School of Mathematics and Statistics,
Nanjing University of Information Science \& Technology, Nanjing, Jiangsu 210044, P.\,R. China}
\email{zhangy2016@nuist.edu.cn}

\date{\today}

\nc{\wtrb}{cocycle weighted Rota-Baxter\xspace}
\nc{\wtrey}{cocycle weighted Reynolds\xspace}
\nc{\relwtrb}{relative \wtrb}
\nc{\relwtrey}{relative \wtrey}
\nc{\Relwtrey}{Relative \wtrey}

\begin{abstract}
Unifying various generalizations of the important notions of  Reynolds operators, the  relative cocycle weighted Reynolds  operators are studied. Here cocycle weighted means the weight of the operators is given by a 2-cocycle rather than by a scaler as in the classical case. We show that the operators and 2-cocycles uniquely determine each other. We further give a characterization of relative cocycle weighted Reynolds operators  in the context of pre-Lie algebras.   Using a method of Liu,  we construct an explicit graded Lie algebra whose Maurer-Cartan elements are given by a relative cocycle weighted Reynolds operator. This allows us to construct the cohomology for a relative cocycle weighted Reynolds operator. This cohomology can also be seen  as the cohomology of a certain pre-Lie algebra  with coefficients in a suitable representation. Then we consider formal deformations of relative cocycle weighted Reynolds operators from cohomological points of view.  Finally, we introduce the notation of  NS-pre-Lie algebras and show  NS-pre-Lie algebras naturally induce pre-Lie algebras and $L$-dendriform algebras.
\end{abstract}

\subjclass[2020]{
17A32, 
17B38, 
17B99 
}

\keywords{Reynolds operator;  Cohomology;  Formal deformation; NS-pre-Lie algebra.}

\maketitle

\tableofcontents

\vspace{-1cm}

\setcounter{section}{0}

\allowdisplaybreaks

\section{Introduction}

\subsection{Reynolds operators}
The origin of Reynolds operators may be found in the study of fluctuation theory in fluid dynamics which was first introduced by  Reynolds \cite{Rey95} in 1895. More precisely, a Reynolds operator is  a linear operator $P$ on an algebra $R$ of functions satisfying the the approximation identity
\begin{align*}
P(uv)=P(u)P(v)+P((u-P(u))(v-P(v))) \, \text{ for all }\, u, v\in R.
\end{align*}
In 1951,  Kamp$\mathrm{\acute{e}}$ de F$\mathrm{\acute{e}}$riet~\cite{KF}  coined the term Reynolds operator in order to study the operator as a mathematical subject in general for an extended period of time. Later, Dubreil-Jacotin~\cite{DJ57} rewrite the  approximation relation as follows:
\begin{align*}
P(u)P(v)=P(uP(v))+P(P(u)v)-P(P(u)P(v)) \, \text{ for all } \, u, v\in R,
\end{align*}
which is the  Reynolds identity we commonly use.
Since then, researchers found that Reynolds operators have a close  connection with combinatorics, geometry and operads, probability theory, algebras of finitely valued functions, operators on function spaces, and rational $G$-modules~\cite{A11, Chu05,  PN15, Rot64, ZtGG}.

\subsection{Rota-Baxter operators and generalizations }
An operator closely related to the Reynolds operator is the Rota-Baxter operator, which is defined by a  linear operator $P:R\rightarrow R$ that satisfies the Rota-Baxter equation
\begin{align*}
P(u)P(v)=P(uP(v)+P(u)v+\lambda uv) \, \text{ for }\, u, v\in R.
\end{align*}
The Rota-Baxter operator originated from the probability study of  Baxter~\cite{B60}, is in order to understand Spitzer's identity in fluctuation theory. Parallel to the fact that a differential algebra could be considered as an algebraic abstraction of differential equations, the Rota-Baxter algebra could be regarded as an algebraic framework of the integral analysis~\cite{Gub}.

The notion of a Rota-Baxter operator on a pre-Lie algebra was introduced by Hou and Bai~\cite{LHB07} for the study of  bialgebra theory of pre-Lie algebras and $L$-dendriform algebras. More precisely, let $(\mathfrak{g}, \cdot)$ be a pre-Lie algebra and $(V;  \mathcal{L}, \mathcal{R})$ be a representation of it. A linear map $K : V \rightarrow \mathfrak{g}$ is a Rota-Baxter operator on $\mathfrak{g}$ with respect to the representation on $V$ if it satisfies
\begin{align}\label{relative-rota-identity}
    Ku \cdot Kv = K ( \mathcal{L}_{Ku} v + \mathcal{R}_{Kv} u) ~ \text{ for } u, v \in V.
\end{align}
Such operators can be viewed as an analog of a Poisson structure.  Generally, Rota-Baxter operators can be defined on algebraic operads, which give rise to the splitting of operads \cite{BBGN13, PBG17}. See \cite{Gub} for more details.

The interaction between studies in pure mathematics and mathematical physics has long been a rich source of inspiration that benefited both fields. Involving a  closed 3-form on a  smooth manifold, \u{S}evera and Weinstein~\cite{SW01} proposed a twisted Poisson structure, which is described as a Dirac structure in suitable Courant algebroids. By analogy with twisted Poisson structures, Uchino~\cite{U08} proposed a new operator--twisted Rota-Baxter operators, which is a natural generalization of Rota-Baxter operators. To our surprise, the twisted Rota-Baxter operators induce the NS-algebra~\cite{L04} which is a twisted version of dendriform algebras.

 Quite recently the discovery of some new algebraic and geometric structures in physics and
mathematics has led to a renewed interest in the study of twisted or relative algebraic structures and their deformations and homotopy theory~\cite{Das,  DG21, HS21, HS212, LST, THS21, WZ21}. One of the main motivations comes from the deformations and  cohomologies of $\mathcal{O}$-Operators. In~\cite{TBGS19}, by Voronov's derived bracket,  Tang, Bai, Guo and Sheng construct a graded Lie algebra, whose Maurer-Cartan elements are Rota-Baxter operators on Lie algebras. Strongly motivated by this pioneering work, Das~\cite{Da20} introduced twisted Rota-Baxter operators on Lie algebras and considered NS-Lie algebras that are related to twisted Rota-Baxter operators. In the same year, two papers ~\cite{CHMM, HS21} independently introduced another generalization of the twisted Rota-Baxter operator, called the twisted Rota-Baxter operator (twisted $\mathcal{O}$-operator) on 3-Lie algebras,  which produces a 3-Lie algebra with a representation on it.
In~\cite{L20}, Liu showed  that certain twisting transformations on (quasi-)twilled
pre-Lie algbras can be characterized by the solutions of Maurer-Cartan equations of the associated differential graded Lie algebras ($L_{\infty}$-algebras). Furthermore, he showed that Rota-Baxter operators and twisted
Rota-Baxter operators are solutions of the Maurer-Cartan equations.

In this paper, we study a relative cocycle weighted Reynolds  operator$^{1}$ on pre-Lie algebras.\footnotetext{1. We thank professor Li Guo for suggesting this notation.}
Using a method of Liu~\cite{L20}, we construct an explicit graded Lie algebra whose Maurer-Cartan elements are given by relative cocycle weighted Reynolds  operators, which make it possible  to construct the cohomology for a relative cocycle weighted Reynolds  operator. We emphasize that this cohomology can also be seen  as the cohomology of a certain pre-Lie algebra induced by the relative cocycle weighted Reynolds  operator with coefficients in a suitable representation.
As an application, we study deformations of a twisted  Rota-Baxter operator.
We show that the infinitesimal in a formal deformation of a relative cocycle weighted Reynolds  operator is a $1$-cocycle in its cohomology.   The classical deformation theory of Gerstenhaber to relative cocycle weighted Reynolds operators are also discussed.
Finally, we introduce the concept of an NS-pre-Lie algebra. We prove that NS-pre-Lie algebras split pre-Lie algebras and serve as the underlying structure of a twisted  Rota-Baxter operator.
It should be pointed  out that  NS-pre-Lie algebras naturally induce $L$-dendriform algebras.
Further study on NS-pre-Lie algebras is postponed to a forthcoming article.

{\bf Structure of the Paper.}  In Section~\ref{sec:RCWR}, we first recall the concept of a \relwtrey operator on pre-Lie algebras. Then we give a new \relwtrey operator on pre-Lie algebras.

In Section~\ref{sec:MCC}, we construct an explicit graded Lie algebra whose Maurer-Cartan elements are given by \relwtrey operators. This allows us to construct the cohomology for a   \relwtrey operator $K$. This cohomology can also be seen  as the cohomology of a certain pre-Lie algebra induced by $K$ with coefficients in a suitable representation.

In Section~\ref{sec:deform},
we use the cohomology to study deformations of a \relwtrey operator.
We introduce Nijenhuis elements associated with a twisted  Rota-Baxter operator that are obtained from trivial linear deformations.
We also find a sufficient condition for the rigidity of a twisted Rota-Baxter operator in terms of Nijenhuis elements.

In Section~\ref{sec:NS-prelie},  we introduce the concept of  NS-pre-Lie algebras and find its relation with \relwtrey operators. We show  NS-pre-Lie algebras naturally induce pre-Lie algebras and $L$-dendriform algebras.

\smallskip

\noindent
{\bf Notation.}
Throughout this paper, all vector spaces, (multi)linear maps are over the field $\mathbb{C}$ of complex numbers and all the vector spaces are finite dimensional.

\section{Relative cocycle weighted  Reynolds  operators on pre-Lie algebras}\mlabel{sec:RCWR}

In this section, we first recall some basic notions and properties on pre-Lie algebras that will be used in this paper. We then recall the concept of a \relwtrey operator on pre-Lie algebras from~\cite{L20}. Finally, we give a new \relwtrey operator on pre-Lie algebras.

\subsection{Pre-Lie algebras and their Cohomology}

We start with the background of pre-Lie algebras and their cohomology that we refer the reader to~\cite{B06, L20} for more details.

\begin{defn}
A {\bf pre-Lie algebra} is a vector space $\mathfrak{g}$  together with a binary multiplication
$\cdot: \mathfrak{g}\otimes \mathfrak{g} \rightarrow \mathfrak{g}$ such that, for
 $x, y,z \in \mathfrak{g}$, the associator $(x, y, z) = (x\cdot y)\cdot z-x\cdot(y\cdot z)$
is symmetric in $x, y,$ i.e.
\begin{align*}
(x, y, z) = (y, x, z),\, \text{ or  equivalently},\, (x \cdot y) \cdot z - x \cdot (y \cdot z) = (y \cdot x) \cdot z - y \cdot (x \cdot z).
\end{align*}
\end{defn}

Let $(\mathfrak{g}, \cdot)$ be a pre-Lie algebra. Define a new  multiplication on $\mathfrak{g}$ by setting
\begin{align*}
[x, y]_{\mathfrak{g}} = x \cdot y - y \cdot x, \text{ for } x, y \in \mathfrak{g}.
\end{align*}
Then $(\mathfrak{g}, [,]_{\mathfrak{g}})$ is a Lie algebra, which is called the {\bf sub-adjacent Lie algebra} of $(\mathfrak{g}, \cdot)$ and we denote it by $\mathfrak{g}^c$.  Furthermore, $L: \mathfrak{g} \rightarrow gl(\mathfrak{g})$ with $x \mapsto L_x$, where $L_xy = x \cdot y$, for all $x, y \in \mathfrak{g}$, gives a module of the Lie algebra $\mathfrak{g}^c$ of $\mathfrak{g}$.

\begin{defn}\cite{B08}
Let $(\mathfrak{g}, \cdot)$ be a pre-Lie algebra and $V$ a vector space. Let $\mathcal{L}, \mathcal{R} : \mathfrak{g}\rightarrow gl(V)$ be
two linear maps with $x \mapsto \mathcal{L}_x$ and $x \mapsto \mathcal{R}_x$, respectively. Then the triple $(V; \mathcal{L}, \mathcal{R})$ is called a {\bf representation}
over $\mathfrak{g}$ if
\begin{align*}
 \mathcal{L}_x\mathcal{L}_yu - \mathcal{L}_{x\cdot y} u =& \mathcal{L}_y\mathcal{L}_xu - \mathcal{L}_{y\cdot x}u,\\
\mathcal{ L}_x\mathcal{R}_yu - \mathcal{R}_y\mathcal{L}_xu =& \mathcal{R}_{x\cdot y}u - \mathcal{R}_y\mathcal{R}_xu,
\end{align*}
for all  $x, y \in \mathfrak{g}, u \in V$.
\end{defn}

Let $(\mathfrak{g}, \cdot)$ be a pre-Lie algebra and  $(V; \mathcal{L}, \mathcal{R})$ be a representation of it. The  {\bf cohomology} of $\mathfrak{g}$ with coefficients in $V$ is the
cohomology of the cochain complex $\{ C^*(\mathfrak{g}, V), \partial \}$, where $C^n (\mathfrak{g}, V) = $Hom $ ( \mathfrak{g}^{\otimes n}, V ), n \geq 0$, and the
coboundary operator $\partial:C^n(\mathfrak{g}, V ) \rightarrow C^{n+1}(\mathfrak{g}, V )$ given by
\begin{eqnarray*}
&&(\partial f)(x_1, ... , x_{n+1})\\
&=&\sum^{n}_{i=1}(-1)^{i+1}\mathcal{L}_{x_i}f(x_1, ... , \hat{x}_i, ... , x_{n+1}) + \sum^{n}_{i=1}(-1)^{i+1}\mathcal{R}_{x_{n+1}}f(x_1, ... ,\hat{x}_i, ... ,  x_{n}, x_i)\\
&&-\sum^{n}_{i=1}(-1)^{i+1}f(x_1,..., \hat{x}_i, ... , x_n,  x_i\cdot x_{n+1})+\sum_{1\leq i< j\leq n} (-1)^{i+j}f([x_i, x_j]_\mathfrak{g}, x_1, ... , \hat{x}_i, ... , \hat{x}_{j}, ... , x_{n+1}),
\end{eqnarray*}
for $x_1, ..., x_{n+1}\in \mathfrak{g}$. The corresponding cohomology groups are denoted by $H^*(\mathfrak{g}, V).$

A permutation $\sigma \in \mathbb{S}_n$ is called an $(i, n-i)${\bf -unshuffle} if
\begin{align*}
\sigma(1) < \cdots < \sigma(i) \text{ and } \sigma(i + 1) < \cdots <
\sigma(n).
\end{align*}
If $i = 0$ or $i = n$, we assume $\sigma = \id$. The set of all $(i, n-i)$-unshuffles will be denoted by $\mathbb{S}_{(i,n-i)}$. By a similarly way, the notion of an $(i_1, ... , i_k)$-unshuffle and the set $\mathbb{S}_{(i_1, ..., i_k)}$ can also be defined.

Let $\mathfrak{g}$ be a vector space. We consider the graded vector space
 \begin{align*}
C^{\ast}(\mathfrak{g}, \mathfrak{g}) = \bigoplus_{n\geq1}C^n(\mathfrak{g}, \mathfrak{g}) =
\bigoplus_{n\geq1} \Hom (\wedge^{n-1}\mathfrak{g} \otimes \mathfrak{g}, \mathfrak{g}).
\end{align*}
It was shown in \cite{CL01} that $C^{\ast}(\mathfrak{g}, \mathfrak{g})$ equipped with the {\bf Matsushima-Nijenhuis bracket}
\begin{eqnarray*}
[P, Q]^{MN} := P \diamond Q - (-1)^{pq}Q \diamond P \tforall P \in C^{p+1}(\mathfrak{g}, \mathfrak{g}), Q \in C^{q+1}(\mathfrak{g}, \mathfrak{g}),
\end{eqnarray*}
is a graded Lie algebra(gLa), where $P \diamond Q \in C^{p+q+1}(\mathfrak{g}, \mathfrak{g})$ is given by
\begin{eqnarray*}
&&P \diamond Q(x_1, ... , x_{p+q+1})\\
&=& \sum_{\sigma\in \mathbb{S}_{(q, 1, p-1)}}\mathrm{sgn}(\sigma)P(Q(x_{\sigma(1)}, ... , x_{\sigma(q)}, x_{\sigma(q+1)}), x_{\sigma(q+2)}, ...  , x_{\sigma(p+q)}, x_{p+q+1})\\
&&+(-1)^{pq} \sum_{\sigma\in \mathbb{S}_{(p, q)}}\mathrm{sgn}(\sigma) P(x_{\sigma(1)}, ... , x_{\sigma(p)}, Q(x_{\sigma(p+1)}, ... , x_{\sigma(p+q)}, x_{p+q+1})).
\end{eqnarray*}
In particular, $\pi\in\Hom(\ot^2\mathfrak{g}, \mathfrak{g})$ defines a pre-Lie algebra if and only if $[\pi, \pi]^{MN}= 0$. If $\pi$ is a pre-Lie algebra, then $d_\pi( f ):= [\pi, f]^{MN}$ is a graded derivation of the gLa $(C^{\ast}(\mathfrak{g}, \mathfrak{g}), [-, -]^{MN})$ satisfying $d_\pi \circ d_\pi = 0$, which implies  that $(C^{\ast}(\mathfrak{g}, \mathfrak{g}), [-, -]^{MN}, d_\pi)$ is a dgLa.

\subsection{ Relative cocycle weighted (RCW) Reynolds  operators}

Let $(\mathfrak{g}, \cdot)$ be a pre-Lie algebra and $(V; \mathcal{L}, \mathcal{R})$ be a representation of it.
Suppose $H \in C^2 (\mathfrak{g},V)$ is a {\bf $2$-cocycle} in the  cochain complex, i.e., $H : \mathfrak{g} \otimes \mathfrak{g} \rightarrow  V$ is a  bilinear map satisfying
\begin{small}
\begin{eqnarray*}
\mathcal{L}_xH(y, z) - \mathcal{L}_yH(x, z) + \mathcal{R}_zH(y, x)-\mathcal{R}_zH(x, y) - H(y, x\cdot z) +H(x, y\cdot z)- H([x, y]_\mathfrak{_g}, z) =0
\end{eqnarray*}
for $x, y, z \in \mathfrak{g}.$
\end{small}
We now give the following notion generalizating Reynolds operators and their generalizations~\cite{GGL,ZtGG}. It was called an $H$-twisted  Rota-Baxter operator on pre-Lie algebras in~\cite{L20}. We adapt the new term since the term $H(Ku,Kv)$ resembles a Reynolds operator rather than a Rota-Baxter operator.
\begin{defn}\label{def:reypre}
Let $H$ be a $2$-cocycle. A linear map $K : V \rightarrow \mathfrak{g}$ is called a {\bf \relwtrey operator} if $K$ satisfies
\begin{eqnarray*}
Ku \cdot Kv=K(\mathcal{L}_{Ku}v+\mathcal{R}_{Kv}u+H(Ku, Kv)) \tforall u, v\in V.
\end{eqnarray*}
\end{defn}
The  Rota-Baxter operator on pre-Lie algebras studied in~\cite{LHB07} is a \relwtrey operator with weight $H = 0$.

\begin{exam}
Consider  3-dimensional pre-Lie algebra $(\mathfrak{g}, \cdot)$ given with
respect to a basis $\{ e_1, e_2, e_3\}$ whose nonzero products are given as follows:
\begin{eqnarray*}
e_3 \cdot e_3=e_2.
\end{eqnarray*}
Then  $K=\left(
\begin{array}{ccc}
 a_{11} & a_{12}  & a_{13}  \\
  a_{21} & a_{22} & a_{23} \\
   a_{31} & a_{32} & a_{33} \\
   \end{array}
   \right)$
is a \relwtrey operator on $(\mathfrak{g}, \cdot)$ with
respect to the regular representation if and only if
\begin{eqnarray*}
&&Ke_i \cdot Ke_j = K (K e_i \cdot e_j+ e_i \cdot Ke_j +H(Ke_i,Ke_j)), ~ \text{ for } i, j=1, 2, 3.
\end{eqnarray*}
For $i, j=1, 2, 3$, suppose that $H: \mathfrak{g}\ot \mathfrak{g}\rightarrow \mathfrak{g}$ defined by
\begin{eqnarray*}
 H(e_i, e_i)=\left\{
               \begin{array}{ll}
                 e_3, & \hbox{if $i=j=3$;} \\
                 0, & \hbox{otherwise.}
               \end{array}
             \right.
\end{eqnarray*}
A directly calculation shows  that $H$ is a $2$-cocycle.
We have
$$K e_1 \cdot Ke_1=(a_{11}e_1 + a_{21}e_2 + a_{31}e_3)\cdot( a_{11}e_1 + a_{21}e_2 + a_{31}e_3) = a^2_{31}e_2.$$
Moreover,
\begin{eqnarray*}
&&K (K e_1\cdot e_1 +e_1\cdot Ke_1+H(Ke_1, Ke_1))\\
&=&a^2_{31}K (e_3)\\
&=&a^2_{31}a_{13}e_1+ a^2_{31}a_{23}e_2+a^2_{31}a_{33}e_3.
\end{eqnarray*}
Then  it follows from $K e_1 \cdot Ke_1=K (K e_1\cdot e_1 +e_1\cdot Ke_1+H(Ke_1, Ke_1))$  that
\begin{eqnarray*}
&&a^2_{31}a_{13}=0,~~ a^2_{31}=a^2_{31}a_{23},~~ a^2_{31}a_{33}=0.
\end{eqnarray*}
By considering other choices of $e_i$ and $e_j$, we obtain
\begin{eqnarray*}
&&a^2_{32}a_{13}=0,~~ a^2_{32}=a^2_{32}a_{23},~~ a^2_{32}a_{33}=0,\\
&& a^2_{33}a_{13} + 2a_{33}a_{12}=0,~~a^2_{33}=a^2_{33}a_{23}+2a_{33}a_{22},~~a^3_{33}+2a_{33}a_{32}=0,\\
&& a_{31}a_{32}a_{13}=0,~~a_{31}a_{32}=a_{31}a_{32}a_{23},~~a_{31}a_{32}a_{33}=0,\\
&& a_{31}a_{33}a_{13}+a_{31}a_{12}=0, a_{31}a_{33}=  a_{31}a_{33}a_{23}+a_{31}a_{22},~~a_{31}a^2_{33}+a_{31}a_{32}=0,\\
&& a_{32}a_{33}a_{13}+a_{32}a_{12}=0, a_{32}a_{33}=  a_{32}a_{33}a_{23}+a_{32}a_{22},~~a_{32}a^2_{33}+a^2_{32}=0.
\end{eqnarray*}
Summarize the above discussion.
If $a_{31}=a_{32}=a_{33}=0$,  then any
$K=\left(
\begin{array}{ccc}
a_{11} & a_{12} & a_{13}  \\
a_{21} & a_{22} & a_{23} \\
0 & 0 & 0 \\
\end{array}
\right)$  is a  \relwtrey operator  on $(\mathfrak{g}, \cdot)$ with
respect to the regular representation.
\end{exam}

In \cite{L04}, Leroux introduced the notion of $TD$-operators on a unital associative algebra $A$ with unit 1. That is, the
linear map $ \gamma: A \rightarrow A$ is said to be a $TD$-operator if,
\begin{eqnarray*}
\gamma(x)\gamma(y) = \gamma(\gamma(x)y + x\gamma(y) - x\gamma(1)y) \tforall x, y\in A.
\end{eqnarray*}

Motivated by the concept of $TD$-operators, we introduce the notion of a $D$-Reynolds operator on pre-Lie algebras.

\begin{defn}
Let $(\mathfrak{g}, \cdot)$  be a unital pre-Lie algebra with unit $1$ and $D: \mathfrak{g} \rightarrow \mathfrak{g}$ be a linear map.
The linear map $K: \mathfrak{g} \rightarrow \mathfrak{g}$ is said to be a {\bf $D$-Reynolds operator} if,
\begin{eqnarray*}
K(x)\cdot K(y) = K(K(x)\cdot y+x\cdot K(y)- (K(x)\cdot D(1))\cdot K(y))   \tforall x, y\in \mathfrak{g}.
\end{eqnarray*}

\end{defn}
Motivated by the concept of  a Reynolds operator of weight $\lambda$ introduced by Zhang, Gao and  Guo~\cite{ZtGG},  the notion of a Reynolds operator of weight $\lambda$ on pre-Lie algebras can be proposed in the same way.
\begin{defn}
Let $(\mathfrak{g}, \cdot)$  be a unital pre-Lie algebra with unit $1$ and $\lambda\in \mathbb{C}$.
The linear map $K: \mathfrak{g} \rightarrow \mathfrak{g}$ is said to be a  Reynolds operator of weight $\lambda$ if,
\begin{eqnarray}\label{defn: Reynolds operator}
K(x)\cdot K(y) = K(K(x)\cdot y+x\cdot K(y)+ \lambda(K(x)\cdot K(y))   \tforall x, y\in \mathfrak{g}.
\end{eqnarray}
A pre-Lie algebra $(\mathfrak{g}, \cdot)$ with a Reynolds operator  $K$ of weight $\lambda$ is called a
Reynolds pre-Lie algebra of weight $\lambda$, denoted by $(\mathfrak{g}, \cdot, K)$.
\end{defn}

\begin{remark}
If $D(1)=-\lambda$, then  the $D$-Reynolds operator is a Reynolds operator of weight $\lambda$. Particularly, if $D(1)=1$, then  a $D$-Reynolds operator is a Reynolds operator.

\end{remark}
\begin{prop} \label{Prop: new pre-Lie algebra}
	Let $(\frakg, \cdot, K)$ be a Reynolds pre-Lie algebra of weight $\lambda$. Define a new binary operation as:
\begin{eqnarray}\label{defn: Reynolds pre-Lie algebra}
x\star  y:=x\cdot   K(y)+K(x)\cdot  y+\lambda (K(x)\cdot  K(y))
\end{eqnarray}
for any $x, y\in \frakg $. Then
\begin{itemize}

\item[(i)] $K (x)\cdot K( y)= K (x\star  y)$;

\item[(ii)]      $(\frakg ,\star )$ is a new  pre-Lie algebra;

    \item[(iii)] the triple  $(\frakg ,\star ,K)$ also forms a Reynolds pre-Lie algebra  of weight $\lambda$;

\item[(iv)] the map $K:(\frakg ,\star , K)\rightarrow (\frakg,\cdot, K)$ is a  morphism of Reynolds pre-Lie  algebras.
    \end{itemize}
	\end{prop}
\begin{proof}
(i) It follows directly from (\ref{defn: Reynolds operator}).

(ii) For $x, y, z\in \frakg$, by (\ref{defn: Reynolds operator}) and (\ref{defn: Reynolds pre-Lie algebra}), we have
\begin{eqnarray*}
&& (x\star  y)\star z-x\star (y\star  z)-(y\star  x)\star z+y\star (x\star  z)\\
&=&  K (x\star  y)\cdot z+(x\star  y)\cdot K(z)+\lambda(K(x\star  y)\cdot K(z))-K(x)\cdot (y\star  z)\\
&&-x\cdot K(y\star  z)-\lambda(K(x)\cdot K(y\star  z))-K (y\star  x)\cdot z-(y\star  x)\cdot K(z)\\
&&-\lambda(K(y\star  x)\cdot K(z))+K(y)\cdot (x\star  z)+y\cdot K(x\star  z)+\lambda(K(y)\cdot K(x\star  z))\\
&=&(K (x)\cdot K(y))\cdot z+(x\star  y)\cdot K(z)+\lambda((K (x)\cdot K(y))\cdot K(z))-K(x)\cdot (y\star  z)\\
&&-x\cdot (K (y)\cdot K(z))-\lambda(K(x)\cdot (K (y)\cdot K(z)))-(K (y)\cdot K(x))\cdot z-(y\star  x)\cdot K(z)\\
&&-\lambda((K (y)\cdot K(x))\cdot K(z))+K(y)\cdot (x\star  z)+y\cdot (K (x)\cdot K(z))+\lambda(K(y)\cdot (K (x)\cdot K(z)))\\
&=&(K (x)\cdot K(y))\cdot z+(x\cdot   K(y)+K(x)\cdot  y+\lambda (K(x)\cdot  K(y)))\cdot K(z)+\lambda((K (x)\cdot K(y))\cdot K(z))\\
&&-K(x)\cdot (y\cdot   K(z)+K(y)\cdot  z+\lambda (K(y)\cdot  K(z)))-x\cdot (K (y)\cdot K(z))-\lambda(K(x)\cdot (K (y)\cdot K(z)))\\
&&-(K (y)\cdot K(x))\cdot z-(y\cdot   K(x)+K(y)\cdot  x+\lambda (K(y)\cdot  K(x)))\cdot K(z)-\lambda((K (y)\cdot K(x))\cdot K(z))\\
&&+K(y)\cdot (x\cdot   K(z)+K(x)\cdot  z+\lambda (K(x)\cdot  K(z)))+y\cdot (K (x)\cdot K(z))+\lambda(K(y)\cdot (K (x)\cdot K(z)))\\
&=&0.
\end{eqnarray*}
Thus $(\frakg ,\star )$ is a   pre-Lie algebra.

(iii) For $x, y\in \frakg$, by  (\ref{defn: Reynolds pre-Lie algebra}), we have
\begin{eqnarray*}
 K(x)\star K(y)&=&  K(x)\cdot   K^2(y)+K^2(x)\cdot  K(y)+\lambda (K^2(x)\cdot  K^2(y))\\
 &=& K(x\star   K(y)+K(x)\star  y+\lambda (K(x)\star  K(y))),
\end{eqnarray*}
which implies that $K$ is a Reynolds operator of weight $\lambda$ on the pre-Lie algebra $(\frakg ,\star)$.

(iv) $K$ is a pre-Lie algebra homomorphism. Moreover, $K$ commutes with itself. Therefore,
$K$ is a Reynolds pre-Lie algebra homomorphism from $(\frakg ,\star, K)$ to $(\frakg ,\cdot, K)$.
\end{proof}

Next, we illustrate the relationship between Reynolds operators of weight $\lambda$  and derivations on a pre-Lie
algebra.

\begin{prop}
Let $(\frakg, \cdot, K)$ be a Reynolds pre-Lie algebra of weight $\lambda$. If $K$ is
invertible, then $(K^{-1}+\lambda \id): \frakg \rightarrow \frakg$ is a derivation on the pre-Lie algebra $\frakg$, where $\id$ is the identity
operator.
\end{prop}
\begin{proof}
Let $K : \frakg \rightarrow \frakg $ be an invertible Reynolds operator of weight $\lambda$. By (\ref{defn: Reynolds operator}), we have
\begin{eqnarray*}
K^{-1}(x\cdot y) = x\cdot K^{-1}(y)+K^{-1}(x)\cdot y+ \lambda(x\cdot y), \forall x, y \in \frakg
\end{eqnarray*}
which implies that, for all $x, y, z \in \frakg$,
\begin{eqnarray*}
(K^{-1}+\lambda \id)(x\cdot y) = x\cdot (K^{-1}+\lambda \id)(y)+(K^{-1}+\lambda \id)(x)\cdot y, \forall x, y \in \frakg
\end{eqnarray*}
This shows that $(K^{-1}+\lambda \id): \frakg \rightarrow \frakg$ is a derivation on the pre-Lie algebra $\frakg$.
\end{proof}
Conversely, we can derive a Reynolds operator of weight $\lambda$ on a pre-Lie algebra from a derivation.
\begin{prop}
Let $D : \frakg \rightarrow \frakg $ be a derivation on a pre-Lie algebra $(\frakg, \cdot)$. If $D-\lambda \id$ is
invertible, then $(D-\lambda \id)^{-1}: \frakg \rightarrow \frakg$ is a Reynolds operator of weight $\lambda$.
\end{prop}
\begin{proof}
Let $D : \frakg \rightarrow \frakg $ be a derivation on a pre-Lie algebra $(\frakg, \cdot)$, for all $u, v\in \frakg$, we
have
\begin{eqnarray*}
 (D-\lambda \id)(u\cdot v) = u\cdot (D-\lambda \id)(v)+(D-\lambda \id)(u)\cdot v+ \lambda(u\cdot v).
\end{eqnarray*}
For convenience, we denote $K = D - \lambda \id $. If $K$ is invertible, we put $Ku = x, Kv = y$,
we get
\begin{eqnarray*}
K^{-1} (x)\cdot K^{-1}(y) = K^{-1}(K^{-1}(x)\cdot y+x\cdot K^{-1}(y)+ \lambda(K^{-1}(x)\cdot K^{-1}(y))),
\end{eqnarray*}
which implies that $K^{-1}$ is a Reynolds operator of weight $\lambda$. The proof is finished.
\end{proof}

\begin{exam}
Let $(\mathfrak{g}, \cdot)$  be a unital pre-Lie algebra with unit $1$ and $D: \mathfrak{g} \rightarrow \mathfrak{g}$ be a linear map.
If the linear map $K: \mathfrak{g} \rightarrow \mathfrak{g}$ is  a $D$-Reynolds operator, take $H(x, y):=-(x\cdot D(1))\cdot y$, it is easy to check that $H$ is a $2$-cocycle. Then
the linear map $K: \mathfrak{g} \rightarrow \mathfrak{g}$ is a \relwtrey operator with
respect to the regular representation.
\end{exam}

\begin{exam}
Let $(\mathfrak{g}, \cdot)$  be a unital pre-Lie algebra with unit $1$ and $f:V\rightarrow \mathfrak{g}$ be a $\mathfrak{g}$-linear surjection.
We fix an element $e \in V$ such that $f(e)=1$.  Take $H(x, y):=-(x\cdot e)\cdot y$.  Then $f:V\rightarrow \mathfrak{g}$ is a \relwtrey operator.  In fact,
\begin{eqnarray*}
&& f(\mathcal{L}_{f(u)}v+\mathcal{R}_{f(v)}u+H(f(u), f(v)))\\
&=& f(u)\cdot f(v)+f(u)\cdot f(v)-f((f(u)\cdot e)\cdot f(v))\\
&=& f(u)\cdot f(v).
\end{eqnarray*}
For instance, $V=\mathfrak{g}\otimes \mathfrak{g}, f =\mu$ and $e=u \otimes u^{-1}$, where $\mu: \mathfrak{g}\otimes \mathfrak{g}\rightarrow \mathfrak{g}, \mu(x\otimes y)=x\cdot y$ is the multiplication
and $u$ is an invertible element.
\end{exam}

\begin{exam}\cite{L20}
Let $(V; \mathcal{L}, \mathcal{R})$ be a representation of a pre-Lie algebra
$(\mathfrak{g}, \cdot)$.  Suppose that $h \in C^1 (\mathfrak{g}, V)$ is an invertible $1$-cochain
in the  cochain complex of $\mathfrak{g}$ with coefficients in $V$. Take $H = - \partial h$. Then
\begin{eqnarray*}
H(Ku, Kv) = (- \partial h)(Ku, Kv) = - \mathcal{L}_{K(u)}v - \mathcal{R}_{K(v)}u + h(Ku\cdot Kv)  \tforall u, v\in V.
\end{eqnarray*}
This shows that $K = h^{-1}: V \rightarrow  \mathfrak{g}$ is a \relwtrey operator.
\end{exam}

\begin{exam}
Let $(\mathfrak{g}, \cdot)$ be a pre-Lie algebra and $N: \mathfrak{g} \rightarrow \mathfrak{g}$ be a Nijenhuis operator on it, that is, $N$ satisfies
\begin{eqnarray*}
Nx \cdot Ny=N(Nx \cdot y + x\cdot Ny - N(x \cdot y)) \tforall x, y \in \mathfrak{g}.
\end{eqnarray*}
In this case the vector space $\mathfrak{g}$ carries a new pre-Lie algebra structure with a deformed operation
\begin{eqnarray}\label{deformed-leib}
x \cdot_N y = Nx \cdot y + x\cdot Ny - N(x \cdot y) \tforall x, y \in \mathfrak{g}.
\end{eqnarray}
This deformed pre-Lie algebra $\mathfrak{g}_N = (\mathfrak{g}, \cdot_N)$ has a representation on $\mathfrak{g}$ by
\begin{align*}
\mathcal{L}_xu:=Nx \cdot u \,\text{ and }\, \mathcal{R}_xu:=u\cdot Nx  \tforall  x \in \mathfrak{g}_N, u \in \mathfrak{g}.
\end{align*}
With this representation, the map $H: (\mathfrak{g}_N)^{\otimes 2} \rightarrow \mathfrak{g},~H(x, y) = -N(x\cdot_N y)$ is a $2$-cocycle in the cohomology of $\mathfrak{g}_N$ with coefficients in $\mathfrak{g}$. Moreover the identity map $\id: \mathfrak{g} \rightarrow \mathfrak{g}_N$ is
a \relwtrey operator.
\end{exam}

 Suppose that $(V'; \mathcal{L}', \mathcal{R}')$ is a representation of the pre-Lie algebra $(\mathfrak{g}', \cdot')$  and $H' \in C^2(\mathfrak{g}', V')$ is a $2$-cocycle.


\begin{defn}
Let $K : V\rightarrow  \mathfrak{g}$, $K' : V'\rightarrow  \mathfrak{g}'$  be  two \relwtrey operators with weight $H$ and $H'$, respectively.
A {\bf morphism of \relwtrey operators} from $K$ to $K'$ consists of a pair $(\phi, \psi)$ of a pre-Lie algebra morphism $\phi: \mathfrak{g} \rightarrow \mathfrak{g}'$ and a linear map $\psi: V \rightarrow V'$ satisfying
\begin{align*}
\phi\circ K=K'\circ \psi, \,
\psi(\mathcal{L}_xu)=\mathcal{L}'_{\phi(x)}\psi(u),\, \psi(\mathcal{R}_xu)=\mathcal{R}'_{\phi(x)}\psi(u),\,
\psi\circ H = H'\circ(\phi\otimes \phi),
\end{align*}  for  all $x\in \mathfrak{g}, u \in V$.
\end{defn}

Let $H$ be a $2$-cocycle $H$ in the  cochain complex of $\mathfrak{g}$ with coefficients in $V$.
Then the direct sum $\mathfrak{g} \oplus V$ carries a pre-Lie algebra structure with the multiplication
\begin{eqnarray}\label{eq:semipro}
(x,u)\cdot_{H} (y,v)= (x\cdot y, \mathcal{L}_x v + \mathcal{R}_y u + H(x, y) ) \, \text{ for all } \, x, y\in \mathfrak{g}, u, v \in V.
\end{eqnarray}
Denote this $H$-twisted semidirect product pre-Lie algebra by $\mathfrak{g} \ltimes_H V$. Having this twisted semidirect product in hand, we now characterize  \relwtrey operators by their graph.

\begin{prop}\label{graph-twisted}
A linear map $K : V\rightarrow  \mathfrak{g}$ is a \relwtrey operator  if and only if its graph
$Gr(K ) = \{ (Ku,u)|~ u \in V\}$ is a subalgebra of the $H$-twisted semidirect product pre-Lie algebra $\mathfrak{g} \ltimes_H V$.
\end{prop}
\begin{proof}
It follows from Definition~\ref{def:reypre} and Eq.~(\ref{eq:semipro}).
\end{proof}

\begin{prop}\label{induced-leib}
Let $K : V\rightarrow  \mathfrak{g}$ be a \relwtrey operator. Then the vector space $V$ carries a pre-Lie algebra structure with the  multiplication
\begin{eqnarray*}
u \cdot_K v:= \mathcal{L}_{Ku} v + \mathcal{R}_{Kv} u + H(K u, Kv)\, \text{ for }\, u, v \in  V.
\end{eqnarray*}
\end{prop}
\begin{proof}
It follows from Proposition~\ref{graph-twisted} and the fact that $Gr (K)$ is isomorphic to $V$ as a vector space.
\end{proof}


 We now proceed to construct a new \relwtrey operator.

\begin{prop}\label{prop-new-con}
Let $(V; \mathcal{L}, \mathcal{R})$ be a representation of a pre-Lie algebra $(\mathfrak{g}, \cdot)$. For any $2$-cocycle $H \in C^2(\mathfrak{g}, V)$ and $1$-cochain $h \in C^1 (\mathfrak{g}, V)$, the twisted semidirect product pre-Lie algebras $\mathfrak{g} \ltimes_H V$ and $\mathfrak{g} \ltimes_{H + \partial h} V$ are isomorphic.
\end{prop}

\begin{proof}
Define a map
\begin{align*}
\Psi_h: \mathfrak{g} \ltimes_H V \rightarrow \mathfrak{g} \ltimes_{H + \partial h} V, \quad  (x,u) \mapsto (x,u - h(x))  \tforall  (x,u) \in \mathfrak{g} \ltimes_H V
\end{align*}
 to be an isomorphism of the underlying vector spaces.  Further, we have
\begin{align*}
 \Psi_h((x,u)\cdot_{H} (y,v))
=&\ \Psi_h(x\cdot y,\mathcal{L}_xv+\mathcal{R}_yu+H(x, y))\\
=&\ (x\cdot y,\mathcal{L}_xv + \mathcal{R}_yu + H(x, y) - h(x\cdot y))\\
=&\ (x\cdot y,\mathcal{L}_xv + \mathcal{R}_yu + H(x, y) - \mathcal{L}_xh(y) - \mathcal{R}_yh(x) + \partial h(x, y) )\\
=&\ (x,u-h(x))\cdot_{H +\partial h} (y,v-h(y))\\
=&\ \Psi_h(x,u)\cdot_{H +\partial h} \Psi_h(y,v).
\end{align*}
Then $\Psi_h$ is an isomorphism of pre-Lie algebras.
\end{proof}

\begin{prop}
Let $K : V \rightarrow \mathfrak{g}$ be a \relwtrey operator and $h$ a 1-cochain in $C^1(\mathfrak{g}, V)$.
If the linear map $(\id_V - h \circ K) : V \rightarrow V$ is invertible, then the map $K \circ (\id_V - h \circ K)^{-1} : V \rightarrow \mathfrak{g}$ is a relative cocycle Reynolds operator of weight  $(H + \partial h)$.
\end{prop}

\begin{proof}
Let $Gr(K)$ be a subalgebra of the $H$-twisted semidirect product  $\mathfrak{g} \ltimes_H V$.  Then by Proposition \ref{prop-new-con},  we have
\begin{align*}
\Psi_h (Gr(K))=\{ (K u,u-hK(u))|~u \in V\} \subset \mathfrak{g} \ltimes_{H + \partial h} V
\end{align*}
is a subalgebra. Since the map $(\id_V -h \circ K ): V\rightarrow V$ is invertible, we have $\Psi_h (Gr(K))$ is the graph of the linear map $K \circ (\id_V - h \circ K)^{-1}$. It follows from Proposition \ref{graph-twisted} that the map $K \circ (\id_V - h \circ K)^{-1}$ is a relative cocycle Reynolds  operator of weight $(H + \partial h)$.
\end{proof}

\begin{defn}
Let $K : V \rightarrow \mathfrak{g}$ be a \relwtrey operator. Suppose $B \in C^1(\mathfrak{g}, V)$ is a $1$-cocycle in the cochain complex of $\mathfrak{g}$ with coefficients in $V$. Then $B$ is  {\bf $K$-admissible} if the linear map $(\id_V + B \circ K ) : V \rightarrow V$ is invertible.
\end{defn}

\begin{prop}
Let $K : V \rightarrow \mathfrak{g}$ be a \relwtrey operator and $B \in C^1 (\mathfrak{g}, V)$  a $K$-admissible $1$-cocycle. Then the map $K \circ (\id_V + B \circ K )^{-1}: V\rightarrow g$ is a \relwtrey  operator.
\end{prop}
\begin{proof}
Consider the deformed subspace
\begin{eqnarray*}
\tau_B (Gr(K)):=\{ (K u, u + B\circ K(u))|~u \in V\} \subset \mathfrak{g} \ltimes_{H} V.
\end{eqnarray*}
We have
\begin{align*}
&\ (Ku,u+B\circ K(u))\cdot_H (Kv,v+B\circ K(v))\\
=&\ (Ku \cdot Kv, \mathcal{L}_{Ku} (v+B\circ K(v)) + \mathcal{R}_{Kv} (u+B\circ K(u)) + H(K u, Kv))\\
=&\ (Ku \cdot Kv, \mathcal{L}_{Ku} (v)+ \mathcal{R}_{Kv} (u)+ H(K u, Kv)+\mathcal{L}_{Ku} (B\circ K(v))+\mathcal{R}_{Kv} (B\circ K(u)))\\
=&\ (Ku \cdot Kv, \mathcal{L}_{Ku} (v)+ \mathcal{R}_{Kv} (u)+ H(K u, Kv)+B(Ku \cdot Kv))\quad \text{(by $B$ being a $1$-cocycle)}\\
=&\ (K(\mathcal{L}_{Ku}v+\mathcal{R}_{Kv}u+H(Ku, Kv)), \mathcal{L}_{Ku} (v)+ \mathcal{R}_{Kv} (u)+ H(K u, Kv)+B(Ku \cdot Kv))
\in  \tau_B (Gr(K))
\end{align*}
Then $\tau_B (Gr (K)) \subset \mathfrak{g} \ltimes_{H} V$  is a subalgebra. Moreover, if the bundle map the map $(\id_V + B \circ K )$ is invertible,
then $\tau_B (Gr (K))$ is the graph of the map $K \circ (\id_V + B \circ K )^{-1}$. Thus  the result follows from Proposition \ref{graph-twisted}.
\end{proof}

The \relwtrey operator in above proposition is called the gauge
transformation of $K$ associated with $B$. We denote it by $K_B$.

\begin{prop}
Let $K$ be a \relwtrey operator and $B$ be a $K$-admissible $1$-cocycle. Then
the pre-Lie algebras on $V$ induced by $K$ and $K_B$ are isomorphic.
\end{prop}
\begin{proof}
 We now consider the linear isomorphism $(\id_V + B \circ K) : V \rightarrow V$. Further, for any $u, v \in V$, we have
\begin{align*}
&\ (\id_V + B \circ K )(u)\cdot_{K_B} (\id_V + B \circ K )(v)\\
=&\ \mathcal{L}_{K(u)} (\id_V + B \circ K )(v) +\mathcal{R}_{K (v)}(\id_V + B \circ K )(u) + H(K u, Kv)\\
=&\  \mathcal{L}_{K(u)}v+ \mathcal{R}_{K(v)}u+ \mathcal{L}_{K(u)} (B \circ K(v)) + \mathcal{R}_{K (v)} (B \circ K(u)) + H(K u, Kv)\\
=&\ \mathcal{L}_{K(u)}v + \mathcal{R}_{K (v)}u + B(Ku \cdot Kv) + H(K u, Kv)\\
=&\ u \cdot_K v+B\circ K(u \cdot_K v)\\
=&\ (\id_V + B \circ K)(u \cdot_K v),
\end{align*}
which implies that $(\id_V + B \circ K) : (V, \cdot_K) \rightarrow (V, \cdot_{K_B})$ is a pre-Lie algebra isomorphism.
\end{proof}

\section{  Maurer-Cartan characterization of RCW Reynolds  operators and  Cohomology}\mlabel{sec:MCC}

In this section, we construct a graded Lie algebra for which the solutions of the Maurer-Cartan equation are precisely  the \relwtrey operators. This allows us to construct a cohomology for a \relwtrey operator $K$. This cohomology can also be seen  as  the  cohomology of the pre-Lie algebra $(V, \cdot_K )$ constructed in Proposition \ref{induced-leib} with coefficients in a suitable representation on $\mathfrak{g}$.

\subsection{$L_{\infty}$-algebras and Maurer-Cartan
elements}
The notion of an $L_{\infty}$-algebras was introduced  by Stasheff in \cite{S92}.   Let $V$ be a graded vector space and $TV$ be its tensor algebra over $V$. We define the graded
skew-symmetric algebra over $V$ by
\begin{eqnarray*}
\wedge V := TV/\langle u \otimes v- (-1)^{|u||v|}v \otimes u\rangle,
\end{eqnarray*}
where $|u|$ and $|v|$ denote the degree of homogeneous elements $u, v \in V$. For any homogeneous
elements $v_1, v_2, ..., v_k$ and a permutation $\sigma \in \mathbb{S}_k$, the Koszul signs $\chi(\sigma)$ is defined by
\begin{eqnarray*}
x_{\sigma(1)} \wedge x_{\sigma(2)} \wedge \cdots \wedge x_{\sigma(k)} =\chi(\sigma)x_1 \wedge x_2 \wedge \cdots \wedge x_k.
\end{eqnarray*}
\begin{defn}
An {\bf $L_{\infty}$-algebra} is a graded vector space $\mathfrak{g}$ equipped with a collection of linear
maps $l_k : \wedge^k \mathfrak{g} \rightarrow \mathfrak{g}$ of degree $2-k$ for $k \geq 1$, such that for all $n \geq 1$
\begin{eqnarray*}
\sum_{i+j=n+1}(-1)^i\sum_{\sigma\in \mathbb{S}_{(i, n-i)}} \chi(\sigma)l_{j}(l_i(x_{\sigma(1)}, ... , x_{\sigma(i)}), x_{\sigma(i+1)}, ... , x_{\sigma(n)})=0.
\end{eqnarray*}
where $v_1, v_2, ..., v_n$ are homogeneous elements in $\mathfrak{g}$.
\end{defn}
\begin{defn}
Let $(\mathfrak{g}, \{l_k\}_{k=1}^{\infty})$ be an $L_{\infty}$-algebra.
An element $\alpha \in \mathfrak{g}^1$ is said to be a {\bf Maurer-Cartan element}  if it satisfies
\begin{eqnarray*}
l_1(\alpha)+\frac{1}{2!} l_2(\alpha, \alpha)-\frac{1}{3!} l_3(\alpha, \alpha, \alpha)-\cdots=0.
\end{eqnarray*}
\end{defn}


\subsection{The Cohomology of   RCW Reynolds operators  and  Maurer-Cartan characterizations}
In this subsection, we construct a representation of the pre-Lie algebra $(V, \cdot_{K})$ and give the cohomology of   \relwtrey operators on the pre-Lie algebra $(V, \cdot_K)$.

\begin{prop}\mlabel{prop:rep}
Let $K : V \rightarrow \mathfrak{g}$ be a \relwtrey operator. Define maps $\overline{\mathcal{L}}, \overline{\mathcal{R}} : V \rightarrow gl (\mathfrak{g})$ by
\begin{small}
\begin{align*}
\overline{\mathcal{L}}_ux = Ku \cdot x - K (\mathcal{R}_xu) - KH(Ku, x),~~~ \overline{\mathcal{R}}_ux = x \cdot Ku - K(\mathcal{L}_xu)-KH(x, Ku),
\end{align*}
\end{small}
for $ u\in V$ and $ x\in \mathfrak{g}$. Then $(\mathfrak{g};  \overline{\mathcal{L}}, \overline{\mathcal{R}})$ is a representation of the pre-Lie algebra $(V, \cdot_{K})$.
\end{prop}
\begin{proof}  For $u, v \in V$ and $x \in \mathfrak{g}$ , we have

\begin{eqnarray*}
&&  \mathcal{\overline{L}}_u\mathcal{\overline{L}}_vx - \mathcal{\overline{L}}_{u\cdot_K v} x \\
&=& \mathcal{\overline{L}}_u(Kv \cdot x - K(\mathcal{R}_xv) - KH(Kv, x)) - (K(u\cdot_K v)\cdot x - K(\mathcal{R}_x(u\cdot_K v)) - KH(K(u\cdot_K v), x))\\
&=&Ku \cdot (Kv \cdot x) - Ku \cdot K(\mathcal{R}_xv) - Ku \cdot KH(Kv, x) - K (\mathcal{R}_{Kv \cdot x}u)
+K (\mathcal{R}_{K\mathcal{R}_xv}u) \\
&&+ K (\mathcal{R}_{KH( Kv, x)}u) - KH(Ku, Kv \cdot x)  + KH(Ku,  K(\mathcal{R}_xv)) + KH(Ku, KH( Kv, x))\\
&& - (K(u)\cdot K(v))\cdot x +K(\mathcal{R}_x(u\cdot_K v)) + KH(K(u)\cdot K(v), x))\\
&=&Kv \cdot (Ku \cdot x) - Kv \cdot K(\mathcal{R}_xu) - Kv \cdot KH(Ku, x) - K (\mathcal{R}_{Ku \cdot x}v)
+K (\mathcal{R}_{K\mathcal{R}_xu}v) \\
&&+ K (\mathcal{R}_{KH( Ku, x)}v) - KH(Kv, Ku \cdot x)  + KH(Kv,  K(\mathcal{R}_xu)) + KH(Kv, KH( Ku, x))\\
&& - (K(v)\cdot K(u))\cdot x +K(\mathcal{R}_x(v\cdot_K u)) + KH(K(v)\cdot K(u), x))\quad (\text{by $H$ being a $2$-cocycle.})\\
&=&\mathcal{\overline{L}}_v\mathcal{\overline{L}}_ux - \mathcal{\overline{L}}_{v\cdot_K u}x.
\end{eqnarray*}
Then
 \begin{eqnarray*}
 \mathcal{\overline{L}}_u\mathcal{\overline{L}}_vx - \mathcal{\overline{L}}_{u\cdot_K v} x = \mathcal{\overline{L}}_v\mathcal{\overline{L}}_ux - \mathcal{\overline{L}}_{v\cdot_K u}x.
 \end{eqnarray*}
We also have
\begin{eqnarray*}
&&  \mathcal{ \overline{L}}_u\mathcal{\overline{R}}_vx - \mathcal{\overline{R}}_v\mathcal{\overline{L}}_ux \\
&=& \mathcal{ \overline{L}}_u(x \cdot Kv - K(\mathcal{L}_xv)-KH(x, Kv))-\mathcal{\overline{R}}_v(Ku \cdot x - K(\mathcal{R}_xu) - KH( Ku, x))\\
&=&Ku \cdot(x \cdot Kv) - Ku \cdot K(\mathcal{L}_xv) - Ku \cdot KH(x, Kv) - K(\mathcal{R}_{x \cdot Kv}u) +K (\mathcal{R}_{K\mathcal{L}_xv}u) \\
&&+ K(\mathcal{R}_{KH(x, Kv)}u) - KH(Ku, x \cdot Kv) + KH(Ku,K(\mathcal{L}_xv)) + KH(Ku,KH(x, Kv))\\
&& - (Ku \cdot x)\cdot Kv +  K(\mathcal{R}_xu)\cdot Kv + KH( Ku, x)\cdot Kv + K (\mathcal{L}_{Ku\cdot x}v) - K(\mathcal{L}_{K\mathcal{R}_xu}v) \\
&&- K (\mathcal{L}_{KH( Ku, x)}v) + KH(Ku \cdot x, Kv) - KH(K(\mathcal{R}_xu), Kv) - KH( KH( Ku, x), Kv)\\
&=&x \cdot (Ku \cdot Kv) - K ( \mathcal{L}_x\mathcal{L}_{Ku}v) -  K (\mathcal{L}_x \mathcal{R}_{Kv} u) - K\mathcal{L}_x H(Ku, Kv))  - KH(x, K(u\cdot_{K}v))\\
&& - (x\cdot Ku)\cdot Kv +  K(\mathcal{L}_xu)\cdot Kv + KH(x,  Ku)\cdot Kv + K (\mathcal{L}_{x\cdot Ku}v) - K(\mathcal{L}_{K\mathcal{L}_xu}v) \\
&&- K (\mathcal{L}_{KH( x, Ku)}v) + KH(x\cdot Ku, Kv) - KH(K(\mathcal{L}_xu), Kv) - KH( KH(  x, Ku), Kv)\\
&=&x \cdot K(u \cdot_K v) - K(\mathcal{L}_x(u\cdot_Kv)) - KH(x, K(u\cdot_{K}v))\\
&& - (x\cdot Ku)\cdot Kv +  K(\mathcal{L}_xu)\cdot Kv + KH(x,  Ku)\cdot Kv + K (\mathcal{L}_{x\cdot Ku}v) - K(\mathcal{L}_{K\mathcal{L}_xu}v) \\
&&- K (\mathcal{L}_{KH( x, Ku)}v) + KH(x\cdot Ku, Kv) - KH(K(\mathcal{L}_xu), Kv) - KH( KH(  x, Ku), Kv)\\
&=& \mathcal{\overline{R}}_{u\cdot v}x - \mathcal{\overline{R}}_v\mathcal{\overline{R}}_ux,
\end{eqnarray*}
which implies that
\begin{eqnarray*}
\mathcal{ \overline{L}}_u\mathcal{\overline{R}}_vx - \mathcal{\overline{R}}_v\mathcal{\overline{L}}_ux = \mathcal{\overline{R}}_{u\cdot v}x - \mathcal{\overline{R}}_v\mathcal{\overline{R}}_ux.
 \end{eqnarray*}
 Thus $(\mathfrak{g};  \overline{\mathcal{L}}, \overline{\mathcal{R}})$ is a representation of the pre-Lie algebra $(V, \cdot_{K})$.
\end{proof}

 We will now consider the cohomology of the pre-Lie algebra $(V, \cdot_K)$ with coefficients in the representation $(\mathfrak{g};  \overline{\mathcal{L}}, \overline{\mathcal{R}})$. More precisely, we define
\begin{eqnarray*}
C^n(V, \mathfrak{g}):= \mbox{Hom}( \wedge^{n-1}V\otimes V, \mathfrak{g}),~ \text{ for } n\geq 0.
\end{eqnarray*}
and the differential $\partial_K: C^n(V, \mathfrak{g}) \rightarrow  C^{n+1}(V, \mathfrak{g})$ by
\begin{align*}
&(\partial_Kf)(u_1, ..., u_{n+1})\\
=& \sum_{i=1}^{n}(-1)^{i+1}Ku_i \cdot f(u_1, ..., \hat{u}_i, ..., u_{n+1}) - \sum_{i=1}^{n} (-1)^{i+1} K (\mathcal{R}_{f(u_1, ..., \hat{u}_i, ..., u_{n+1})} u_i)\\
&-\sum_{i=1}^{n}(-1)^{i+1}KH( Ku_i, f(u_1, ..., \hat{u}_i, ..., u_{n+1}))+(-1)^{n+1}f(u_1, ..., u_{n}) \cdot Ku_{n+1}\\
& +(-1)^{n} K (\mathcal{L}_{f(u_1, ..., u_{n})}u_{n+1}) +(-1)^{n}KH( f(u_1, ...,\hat{u}_i, ...,  u_{n}, u_i), Ku_{n+1})\\
&-\sum^{n}_{i=1}(-1)^{i+1}f(u_1,..., \hat{u}_i, ..., u_n,  \mathcal{L}_{Ku_i}u_{n+1} +\mathcal{L}_{Ku_{n+1}}u_i+H(Ku_i, Ku_{n+1}))\\
&+\sum_{1\leq i< j\leq n+1}(-1)^if(\mathcal{L}_{Ku_i}u_j +\mathcal{R}_{Ku_j}u_i+H(Ku_i, Ku_j),  u_1, ..., \hat{u}_i,..., \hat{u}_{j}, ..., u_{n+1})\\
&-\sum_{1\leq i< j\leq n+1}(-1)^if(\mathcal{L}_{Ku_j}u_i +\mathcal{R}_{Ku_i}u_j+H(Ku_j, Ku_i),  u_1, ..., \hat{u}_i,..., \hat{u}_{j}, ..., u_{n+1}),
\end{align*}
for $f\in C^n(V, \mathfrak{g})$ and $u_1, ..., u_{n+1} \in V$. By the definition of $\partial_K$ and Proposition~\ref{prop:rep},  $\partial_K$ is  a special case of $\partial$,  Thus  $\partial_K^2=0$.
Denote by
\begin{align*}
    Z^n (V, \mathfrak{g}) = \{ f \in C^n (V, \mathfrak{g}) ~|~ \partial_K f = 0 \}\, \text{ and }\,
    B^n (V, \mathfrak{g}) = \{ \partial_K g ~|~ g \in C^{n-1}(V, \mathfrak{g}) \}.
\end{align*}
The corresponding cohomology groups
\begin{align*}
    H^n (V, \mathfrak{g}) := \frac{  Z^n (V, \mathfrak{g})  }{B^n (V, \mathfrak{g})}, \text{ for } n \geq 0
\end{align*}
are called the cohomology of the \relwtrey operator  $K$.

Let $(\mathfrak{g}, \cdot)$ be a pre-Lie algebra  and $(V; \mathcal{L}, \mathcal{R})$ be a representation of it.
Denote the Lie bracket on $\mathfrak{g}$ by the map $\mu : \wedge^2 \mathfrak{g} \rightarrow \mathfrak{g}$.
Then the graded vector space $\bigoplus_{n\geq 0}$ Hom$(\wedge^{n-1}V\otimes V, \mathfrak{g})$ carries a graded Lie bracket given by
\begin{eqnarray*}
\llbracket P, Q\rrbracket:= (-1)^{p-1} [[\mu + \mathcal{L}+\mathcal{R}, P]^{MN},Q]^{MN},
\end{eqnarray*}
for $P \in$ Hom$(\wedge^{p-1} V\otimes V, \mathfrak{g}), Q \in$ Hom$(\wedge^{q-1} V\otimes V, \mathfrak{g})$. Explicitly,
\begin{eqnarray*}
&&\llbracket P, Q\rrbracket (u_1, . . . , u_{p+q})\\
 &=&-\sum_{\sigma\in \mathbb{S}_{(q, 1, p-2)}} (-1)^{\sigma} P(\mathcal{L}_{Q(u_{\sigma(1)}, . . . , u_{\sigma(q)}, u_{\sigma(q)})}u_{\sigma(q+1)}, u_{\sigma(q+2)}, . . . , u_{\sigma(p+q-1)}, u_{p+q})\\
 &&+(-1)^{pq}\sum_{\sigma\in \mathbb{S}_{(p-1,q-2, 1)}} (-1)^{\sigma} P(u_{\sigma(1)}, u_{\sigma(2)}, . . . , u_{\sigma(p-1)},\mathcal{ R}_{Q(u_{\sigma(p+1)}, u_{\sigma(p+2)}, . . . , u_{\sigma(p+q-1)}, u_{p+q})}u_{\sigma(p)} )\\
&& +(-1)^{pq}\Bigg(\sum_{\sigma\in \mathbb{S}_{(p, 1, q-2)}} (-1)^{\sigma} Q(\mathcal{L}_{P(u_{\sigma(1)}, . . . , u_{\sigma(p)})}u_{\sigma(p+1)}, u_{\sigma(p+2)}, . . . , u_{\sigma(p+q-1)}, u_{p+q})\\
&& -(-1)^{pq}\sum_{\sigma\in \mathbb{S}_{(q-1, q,  1)}} (-1)^{\sigma} Q(u_{\sigma(1)}, u_{\sigma(2)}, . . . , u_{\sigma(q-1)},\mathcal{ R}_{P(u_{\sigma(q+1)}, u_{\sigma(q+2)}, . . . , u_{\sigma(q+p-1)}, u_{p+q})}u_{\sigma(q)} ) \Bigg)\\
&& -(-1)^{pq} \Bigg(\sum_{\sigma\in \mathbb{S}_{(p, q-1)}} (-1)^{\sigma} P(u_{\sigma(1)}, . . . , u_{\sigma(p)})\cdot Q(u_{\sigma(p+1)}, . . . , u_{\sigma(p+q-1)}, u_{p+q})\\
&&-(-1)^{pq}\sum_{\sigma\in \mathbb{S}_{(q, p-1)}} (-1)^{\sigma} Q(u_{\sigma(1)}, . . . , u_{\sigma(q)})\cdot P(u_{\sigma(q+1)}, . . . , u_{\sigma(p+q-1)}, u_{p+q})\Bigg).
\end{eqnarray*}

Hence, for any $K \in $ Hom$(V, \mathfrak{g})$,
\begin{eqnarray}\label{Maurer-Cartan}
\llbracket K, K\rrbracket(u, v)=2 (K u \cdot K v-K(\mathcal{L}_{K(u)}v +\mathcal{R}_{T (v)} u)), \mbox{ for}~ u, v \in V.
\end{eqnarray}
This shows that Maurer-Cartan elements on this graded Lie algebra are Rota-Baxter operators.

Let $H$ be a 2-cocycle in the cohomology of $\mathfrak{g}$ with coefficients in $V$.  We introduce a ternary degree -1 bracket $\llbracket-,-,-\rrbracket$ on the graded space
$\bigoplus_{n \geq0}$ Hom$(\wedge^{n-1} V\otimes V, \mathfrak{g})$ as follows. For $P \in$ Hom$(\wedge^{p-1} V\otimes V, \mathfrak{g})$, $Q \in $ Hom$(\wedge^{q-1} V\otimes V, \mathfrak{g})$, and $R \in$ Hom$(\wedge^{r-1} V\otimes V, \mathfrak{g})$, we define $\llbracket P, Q, R\rrbracket\in $ Hom$(\wedge^{p+q+r-2} V\otimes V, \mathfrak{g})$ by
\begin{footnotesize}
\begin{align}\label{ternary-brkt}
&\ \llbracket P,Q,R\rrbracket(u_1, . . . , u_{p+q+r-2}\otimes u_{p+q+r-1})\nonumber\\
=&\ \sum_{\sigma\in \mathbb{S}_{(q-1,1, r-1,1, p-2)}}(-1)^{rq+p-1} (-1)^{\sigma} P\Big(H(Q(u_{\sigma(1)}, . . . u_{\sigma(q-1)}\otimes u_{\sigma(q)}), R(u_{\sigma(q+1)}, . . . , u_{\sigma(q+r-1)}\otimes u_{\sigma(q+r)})), \nonumber\\
&\ u_{\sigma(q+r+1)}, ...,u_{\sigma(p+q+r-2)}\otimes u_{p+q+r-1}\Big)\nonumber\\
&\ + \sum_{\sigma\in \mathbb{S}_{(r-1,1, q-1,1, p-2)}}(-1)^{p} (-1)^{\sigma}P\Big(H(R(u_{\sigma(1)}, . . . ,u_{\sigma(r-1)}\otimes  u_{\sigma(r)}),
 Q(u_{\sigma(r+1)}, . . .u_{\sigma(q+r-1)} \otimes u_{\sigma(q+r)})), \nonumber\\
&\ u_{\sigma(q+r+1)}, ...,u_{\sigma(p+q+r-2)}\otimes u_{p+q+r-1}\Big)\nonumber\\
&\ + \sum_{\sigma\in \mathbb{S}_{(p-1,q-1,1, r-1)}}(-1)^{(r+1)(p+q+1)+pq+p+1}
 (-1)^{\sigma} P\Big(u_{\sigma(1)}, ...,u_{\sigma(p-1)}\otimes H(Q(u_{\sigma(p)}, . . . , u_{\sigma(p+q-2)}\otimes u_{\sigma(p+q-1)}), \nonumber\\
&\ R(u_{\sigma(p+q)}, . . . , u_{\sigma(p+q+r-2)}\otimes u_{p+q+r-1}))\Big)\nonumber\\
&\ +\sum_{\sigma\in \mathbb{S}_{(r-1, 1, p-1,1,  q)}} (-1)^{pq+q+1} (-1)^{\sigma} Q\Big(H(R(u_{\sigma(1)}, . . . ,u_{\sigma(r-1)}\otimes u_{\sigma(r)}), P(u_{\sigma(r+1)}, . . . , u_{\sigma(p+r-1)}\otimes u_{\sigma(p+r)})), \nonumber\\
&\ u_{\sigma(p+r+1)}, ...,u_{\sigma(p+q+r-2)} \otimes u_{p+q+r-1}\Big)\nonumber\\
&\ + \sum_{\sigma\in \mathbb{S}_{(p-1,1,  r-1, 1, q-2)}}(-1)^{rp+pq+q}(-1)^{\sigma} Q\Big(H(P(u_{\sigma(1)}, . . . ,u_{\sigma(p-1)} \otimes u_{\sigma(p)}), \nonumber\\
&\ R(u_{\sigma(p+1)}, . . . ,u_{\sigma(p+r-1)} \otimes u_{\sigma(p+r)})),  u_{\sigma(p+r+1)}, ...,u_{\sigma(p+q+r-2)}\otimes u_{p+q+r-1}\Big)\nonumber\\
&\ +\sum_{\sigma\in \mathbb{S}_{(q-1, p, r-1)}}(-1)^{(r-1)(p+q-1)+q}
 (-1)^{\sigma} Q\Big(u_{\sigma(1)}, . . . ,u_{\sigma(q-1)}\otimes H(P(u_{\sigma(q)}, . . . , u_{\sigma(p+q-2)}\otimes u_{\sigma(p+q-1)}), \nonumber\\
&\ R(u_{\sigma(p+q)}, . . . , u_{\sigma(p+q+r-2)}\otimes u_{p+q+r-1}))\Big)\nonumber\\
&\ +\sum_{\sigma\in \mathbb{S}_{(q-1, 1,  p-1, 1, r-2)}}(-1)^{(r-1)(p+q-1)++p+q-1} (-1)^{\sigma}R\Big(H(Q(u_{\sigma(1)}, . . . , u_{\sigma(q-1)}\otimes u_{\sigma(q)}), P(u_{\sigma(q+1)}, . . . , u_{\sigma(p+q-1)}\otimes u_{\sigma(p+q)})), \nonumber\\
&\ u_{\sigma(p+q+1)}, ...,u_{\sigma(p+q+r-2)} \otimes u_{p+q+r-1}\Big)\nonumber\\
&\ +\sum_{\sigma\in \mathbb{S}_{(p-1, 1,  q-1, 1, r-2)}}(-1)^{p(q-1)+(r-1)(p+q-1)+q} (-1)^{\sigma}R\Big(H(P(u_{\sigma(1)}, . . . , u_{\sigma(p-1)}\otimes u_{\sigma(p)}),
Q(u_{\sigma(p+1)}, . . . , u_{\sigma(p+q-1)}\otimes u_{\sigma(p+q)})),\nonumber\\
&\ u_{\sigma(p+q+1)}, ...,u_{\sigma(p+q+r-2)} \otimes u_{p+q+r-1}\Big)\nonumber\\
&\ +\sum_{\sigma\in \mathbb{S}_{(r-1, q-1,1, p-1)}}(-1)^{p+q+1}(-1)^{\sigma}
 R\Big(u_{\sigma(1)}, . . . , u_{\sigma(r-1)}\otimes H(Q(u_{\sigma(r)}, . . . , u_{\sigma(r+q-2)}\otimes u_{\sigma(r+q-1)}),\nonumber\\
&\ P(u_{\sigma(r+q)}, . . . , u_{\sigma(p+q+r-2)}\otimes u_{p+q+r-1}))\Big)\nonumber\\
&\ + \sum_{\sigma\in \mathbb{S}_{(r-1, p-1,1, q-1)}}
 (-1)^{pq+p+q}(-1)^{\sigma}R\Big(u_{\sigma(1)}, . . . , u_{\sigma(r-1)}\otimes H(P(u_{\sigma(r)}, . . . , u_{\sigma(r+p-2)} \otimes u_{\sigma(r+p-1)}),\nonumber\\
&\ Q(u_{\sigma(r+p)}, . . . , u_{\sigma(p+q+r-2)}\otimes u_{p+q+r-1}))\Big).
\end{align}
\end{footnotesize}
The binary bracket $\llbracket -, -\rrbracket$ and the
ternary bracket $\llbracket -, -, -\rrbracket$ are compatible in the sense of $L_{\infty}$-algebra. This follows as $H$ is a 2-cocycle. In
summary, we obtain the following.

\begin{theorem}
Let $(\mathfrak{g}, \cdot)$ be a pre-Lie algebra  and $(V; \mathcal{L}, \mathcal{R})$ be a representation of it,  and $H$ be a 2-cocycle in the cohomology of $\mathfrak{g}$ with
coefficients in $V$. Then the graded vector space $\bigoplus _{n\geq 0}$ Hom$(\wedge^{n-1} V\otimes V, \mathfrak{g})$ is an $L_{\infty}$-algebra with
\begin{eqnarray*}
l_1=0,~~~ l_2=\llbracket -, -\rrbracket,~~~ l_3=\llbracket -,-,-\rrbracket
\end{eqnarray*}
and higher maps are trivial. A linear map $K : V\rightarrow \mathfrak{g}$ is a \relwtrey  operator if and only if
$K \in$ Hom $(V, \mathfrak{g})$ is a Maurer-Cartan element in the above $L_{\infty}$-algebra $(\bigoplus _{n\geq 0}$ Hom$(\wedge^{n-1} V\otimes V, \mathfrak{g}), \llbracket-, -\rrbracket,  \llbracket-,-,-\rrbracket)$.
\end{theorem}
\begin{proof}
 The first part follows from the previous discussions. For the second part, we observe that for any $K\in $ Hom$(V, \mathfrak{g})$,
\begin{eqnarray}\label{ternary-element}
\llbracket K, K, K \rrbracket (u, v) = 6 K (H(K u, K v)).
\end{eqnarray}
Hence from (\ref{Maurer-Cartan}) and (\ref{ternary-element}), we get
\begin{eqnarray*}
(\frac{1}{2}\llbracket K, K\rrbracket+\frac{1}{6}\llbracket K, K, K \rrbracket) (u, v)=(K u \cdot K v-K(\mathcal{L}_{K(u)}v +\mathcal{R}_{T (v)} u))+K (H(K u, K v)).
\end{eqnarray*}
This shows that $K$ is a \relwtrey operator if and only if $K$ is a Maurer-Cartan element in the
$L_{\infty}$-algebra.
\end{proof}
\begin{remark}
In \cite{L20},  Liu considered  a graded Lie algebra associated to a pre-Lie algebra and a bimodule over it. A \relwtrey operator can be characterized by certain solutions of the corresponding Maurer-Cartan equation.  Here we follow the result of Liu and the derived bracket construction of Chapoton and Livernet \cite{CL01} to construct an explicit graded Lie algebra whose Maurer-Cartan elements are precisely \relwtrey operators.
\end{remark}

Let $K$ be a \relwtrey operator. Since $K$ is a Maurer-Cartan element of
the graded Lie algebra $(C^{\ast}(V, g), \llbracket-, -\rrbracket, \llbracket-, -, -\rrbracket)$, it follows from the
graded Jacobi identity that the map:
\begin{eqnarray*}
d_K :C^n(V, g) \rightarrow C^{n+1}(V, g), d_K (f) = \llbracket K, f\rrbracket+\frac{1}{2}\llbracket K, K, f\rrbracket,~~~f\in C^{n}(V, g)
\end{eqnarray*}
is a graded derivation of the graded Lie algebra $(C^{\ast}(V, g), \llbracket-, -\rrbracket, \llbracket-, -, -\rrbracket)$ satisfying
$d_K \circ d_K = 0$.
\begin{prop}
Let $K$ be a \relwtrey operator. Then $\bigoplus_{n\geq 0}$ Hom$(\wedge^{n-1} V\otimes V, \mathfrak{g})$ carries an $L_{\infty}$-algebra with structure maps
\begin{eqnarray*}
l_1(P)=d_K(P),~~\llbracket P, Q\rrbracket_K= \llbracket P, Q\rrbracket+\llbracket K, P, Q\rrbracket,~~\llbracket P, Q,  R\rrbracket_K=\llbracket P, Q,  R\rrbracket
\end{eqnarray*}
and trivial higher brackets. We call this as the {\bf twisted $L_{\infty}$-algebra}  by $K$.
\end{prop}
\begin{theorem}
Let $K$ be a \relwtrey operator. For any linear map $K': V \rightarrow \mathfrak{g}$, the sum
$K + K'$ is a \relwtrey operator if and only if $K'$ is a Maurer-Cartan element in the above
twisted $L_{\infty}$-algebra.
\end{theorem}
\begin{proof}  We have
\begin{eqnarray*}
&& \frac{1}{2}\llbracket K+K', K+K'\rrbracket+\frac{1}{6}\llbracket K+K', K+K', K+K' \rrbracket\\
&=& \frac{1}{2}\llbracket K, K'\rrbracket+\frac{1}{2}\llbracket K', K\rrbracket+\frac{1}{2}\llbracket K', K'\rrbracket+\frac{1}{6}\llbracket K, K, K' \rrbracket\\
&& +\frac{1}{6}\llbracket K, K', K\rrbracket+\frac{1}{6}\llbracket K+K', K, K\rrbracket+\frac{1}{6}[[K, K', K' \rrbracket\\
&& +\frac{1}{6}\llbracket K', K, K' \rrbracket+\frac{1}{6}\llbracket K', K', K \rrbracket+\frac{1}{6}\llbracket K', K', K' \rrbracket\\
&=&\llbracket K, K'\rrbracket+\frac{1}{2}\llbracket K, K, K'\rrbracket+\frac{1}{2}(\llbracket K', K'\rrbracket+\llbracket K, K', K'\rrbracket)+\frac{1}{6}\llbracket K', K', K' \rrbracket\\
&=& d_K(K')+\frac{1}{2}\llbracket K', K'\rrbracket_K+\frac{1}{6}\llbracket K', K', K' \rrbracket_K.
\end{eqnarray*}
This completes the proof.
\end{proof}

\begin{prop}
Let $K$ be a \relwtrey operator. Then for any $f\in Hom(\wedge^{n-1}V\otimes V, \mathfrak{g})$, we have
\begin{eqnarray*}
d_K f=(-1)^{n-1}\partial_K f.
\end{eqnarray*}
\end{prop}
\begin{proof}
 For all $u_1, u_2, ..., u_{n+1} \in  V$, we have
\begin{eqnarray*}
&& \llbracket K,f \rrbracket (u_1, ..., u_{n+1})\\
&=& (-1)^{n-1}\bigg\{ \sum_{i=1}^{n}(-1)^{i+1}Ku_i \cdot f(u_1, ..., \hat{u}_i, ..., u_{n+1}) - \sum_{i=1}^{n} (-1)^{i+1} K (\mathcal{R}_{f(u_1, ..., \hat{u}_i, ..., u_{n+1})} u_i)\\
&&+(-1)^{n+1}f(u_1, ..., u_{n}) \cdot Ku_{n+1} +(-1)^{n} K (\mathcal{L}_{f(u_1, ..., u_{n})}u_{n+1})\\
&&-\sum^{n}_{i=1}(-1)^{i+1}f(u_1,..., \hat{u}_i, ..., u_n,  \mathcal{L}_{Ku_i}u_{n+1} +\mathcal{L}_{Ku_n+1}u_i)\\
&&+\sum_{1\leq i< j\leq n+1}(-1)^if(\mathcal{L}_{Ku_i}u_j,  u_1, ..., \hat{u}_i,..., \hat{u}_{j}, ..., u_{n+1})\\
&&-\sum_{1\leq i< j\leq n+1}(-1)^if(\mathcal{L}_{Ku_j}u_i,  u_1, ..., \hat{u}_i,..., \hat{u}_{j}, ..., u_{n+1}) \bigg\}.
\end{eqnarray*}
 Next from the expression (\ref{ternary-brkt}), we see that
\begin{eqnarray*}
&& \llbracket K, K, f \rrbracket (u_1, ..., u_{n+1})\\
&=&  (-1)^{n-1}~ 2\bigg\{-\sum_{i=1}^{n}(-1)^{i+1}KH( Ku_i, f(u_1, ..., \hat{u}_i, ..., u_{n+1}))\\
&&  +(-1)^{n}KH( f(u_1, ...,\hat{u}_i, ...,  u_{n}, u_i), Ku_{n+1})\\
&&-\sum^{n}_{i=1}(-1)^{i+1}f(u_1,..., \hat{u}_i, ..., u_n, H(Ku_i, Ku_{n+1}))\\
&&+\sum_{1\leq i< j\leq n+1}(-1)^if(H(Ku_i, Ku_j),  u_1, ..., \hat{u}_i,..., \hat{u}_{j}, ..., u_{n+1})\\
&&-\sum_{1\leq i< j\leq n+1}(-1)^if(H(Ku_j, Ku_i),  u_1, ..., \hat{u}_i,..., \hat{u}_{j}, ..., u_{n+1})\bigg\}.
\end{eqnarray*}
Hence, $d_K f=\llbracket K,f \rrbracket+\frac{1}{2}\llbracket K, K, f \rrbracket =(-1)^{n-1}\partial_K f$. This completes the proof.
\end{proof}

\section{ Deformations of \relwtrey operators}\label{sec:deform}

In this section, we first apply the  deformation theory of Gerstenhaber to \relwtrey operators. We then introduce the notation of Nijenhuis elements associated with a \relwtrey operator  arising from trivial linear deformations.
Finally, we give a sufficient condition for rigidity in terms of Nijenhuis elements.

\subsection{Linear deformations}
We first recall some basic concepts and facts about linear deformations that will be used in this paper.

Let $(\mathfrak{g}, \cdot)$ be a pre-Lie algebra, $(V;  \mathcal{L}, \mathcal{R})$ be its representation, and $H \in C^2 (\mathfrak{g}, V)$ be a $2$-cocycle in the  cochain complex. Let $K : V \rightarrow \mathfrak{g}$ be a \relwtrey operator.
A linear map $K_1 : V \rightarrow \mathfrak{g}$ is said to {\bf generate a linear deformation} of the \relwtrey operator $K$ if  the sum $K_t = K +t K_1, t \in \mathbb{C}$ is still a \relwtrey operator. Then we call $K_t = K +t K_1$ is  a {\bf linear deformation} of $K$.

Assome that $K_1$ generates a linear deformation of $K$. Then
\begin{align*}
    K_t u \cdot K_t v = K_t \big(   \mathcal{L}_{ K_t u} v + \mathcal{R}_{K_t v} u + H (K_t u, K_tv) \big), ~ \text{for } u, v \in V,
\end{align*}
which is equivalent to the following conditions
\begin{align}\label{1-cocycle}
Ku \cdot K_1v + K_1u \cdot K v =~& K_1(\mathcal{L}_{K u}v + \mathcal{R}_{K v}u+H(Ku, Kv))\nonumber\\
&+K (\mathcal{L}_{K_1u}v + \mathcal{R}_{K_1v}u+H(K_1u, Kv)+H(Ku, K_1v)),
\end{align}
\begin{align}
K_1u \cdot K_1v
 =K_1(L_{K_1u}v ~&+ R_{K_1v}u+H(Ku, K_1v)+H(K_1u, Kv))+KH(K_1u, K_1v),\\
K_1(H(K_1(u), K_1(v)))=~&0.
\end{align}
It follows from  (\ref{1-cocycle})  that $K_1$ is a $1$-cocycle in the cohomology of $K$. Then $K_1$ induces an element in $H^1_K (V, \mathfrak{g}).$

Two linear deformations $K_t = K +tK_1$  and $K'_t = K +tK'_1$ of $K$ are  {\bf equivalent} if there exists an element $x \in \mathfrak{g}$ such that, for each $t$,
\begin{align*}
\Big(\phi_t = \id_\mathfrak{g} +  t(L_x-R_x),~ \psi_t =\id_V + t(\mathcal{L}_{x}-\mathcal{R}_{x}+H(x, K-))\Big)
\end{align*}
is a morphism of \relwtrey operators from $K_t$ to $K'_t$.

Before proving the theorem, we first make some preliminary observations.
We observe  that $\phi_t = \id_\mathfrak{g} +  t (L_x-R_x)$ is a pre-Lie algebra morphism of $(\mathfrak{g}, \cdot)$ is equivalent to
\begin{eqnarray}\label{alg-map-imply}
(x \cdot y-y \cdot x)\cdot (x\cdot z-z \cdot x) =(y\cdot z)\cdot x=0 ~ \text{ for } y, z\in \mathfrak{g}.
\end{eqnarray}
Further, the conditions $\psi_t(\mathcal{L}_{y}u) = \mathcal{L}_{\phi_t(y)} \psi_t(u)$ and $\psi_t(\mathcal{R}_{y}u) = \mathcal{R}_{\phi_t(y)} \psi_t(u)$, for $y \in \mathfrak{g}, u \in V$ are  equivalent to
\begin{eqnarray}\label{left-action-imply}
\left\{
\begin{aligned}
H(x, K(\mathcal{L}_{y}u)) = \mathcal{L}_{y}H(x, Ku), \\
\mathcal{L}_{x\cdot y-y\cdot x} (\mathcal{L}_{x}u-\mathcal{R}_{x}u+H(x, Ku))= 0,
\end{aligned}
\right. \end{eqnarray}
\begin{eqnarray}\label{right-action-imply} \left\{
\begin{aligned}
H(x, K(\mathcal{R}_{y}u)) = \mathcal{R}_{y}H(x, Ku),\\
\mathcal{R}_{x\cdot y-y\cdot x} (\mathcal{L}_{x}u-\mathcal{R}_{x}u+H(x, Ku))= 0,
\end{aligned}
\right. \end{eqnarray}
respectively.
Similarly, the conditions $\psi_t \circ H = H \circ (\phi_t \otimes \phi_t)$ and $\phi_t \circ K_t = K'_t \circ \psi_t$ are  equivalent to
\begin{eqnarray}\label{h-comp-imply}
\left\{
\begin{aligned}
&\mathcal{L}_{x}H(y, z)-\mathcal{R}_{x}H(y, z)+H(x, KH(y, z)) = H(x\cdot y-y\cdot x, z)+H(y, x \cdot z-z\cdot x),\\
&H(x\cdot y-y\cdot x, x \cdot z-z\cdot x)= 0,
\end{aligned}
\right. \end{eqnarray}
\begin{eqnarray}\label{diffe}
 \left\{
\begin{aligned}
& K_1(u)+x \cdot Ku-Ku \cdot x=K(\mathcal{L}_{x}u-\mathcal{R}_{x}u+H(x, Ku)) + K'_1(u), \\
&x \cdot K_1u-K_1u \cdot x = K'_1 (\mathcal{L}_{x}u-\mathcal{R}_{x}u+H(x, Ku)),
\end{aligned}
\right.
\end{eqnarray}
respectively.

\begin{theorem}
Let $K$ be a \relwtrey operator.
If two linear deformations
$K_t = K +tK_1$  and $K'_t = K +tK'_1$ of $K$ are equivalent, then $K_1$ and $K'_1$ are in the same
cohomology class of $H^1_K (V, \mathfrak{g})$.
\end{theorem}
\begin{proof}
By the first identity in (\ref{diffe}), we have  $K_1(u)-K'_1(u) = \partial_{K}(x)(u)$.
\end{proof}

\begin{remark}
If $K_t$ is further equivalent to the undeformed deformation $K'_t=K$, we call the linear deformation $K_t = K +tK_1$ of a \relwtrey operator $K$ is  {\bf trivial}.
\end{remark}

\subsection{Formal deformations }

Let $\mathbb{C}[[t]]$ be the ring of power series in one variable $t$. For any $\mathbb{C}$-linear space $V$,  let $V [[t]]$ denote the vector space of formal power series in $t$ with coefficients in $V$. If $(\mathfrak{g}, \cdot)$ is a pre-Lie algebra over $\mathbb{C}$, then by $\mathbb{C}[[t]]$-bilinearity, we can extend the pre-Lie bracket on $\mathfrak{g} [[t]]$ . Furthermore, if $(V; \mathcal{L}, \mathcal{R})$ is a representation of the pre-Lie algebra $(\mathfrak{g}, \cdot)$, then there is a representation $(V [[t]];  \mathcal{L}, \mathcal{R})$ of the pre-Lie algebra $\mathfrak{g}[[t]]$. Here, $\mathcal{L}$ and $\mathcal{R}$ are also extended by $\mathbb{C}[[t]]$-bilinearity.
By the same way, the $2$-cocycle $H$ can be extended to a $2$-cocycle on the pre-Lie algebra $\mathfrak{g}[[t]]$ with coefficients in $V[[t]]$. Without loss of confusion, we still denote this $2$-cocycle by $H$.

Let $K : V \rightarrow \mathfrak{g}$ be a  relative cocycle weighted Reynolds
operator and $H$ a $2$-cocycle.
Define the  power series
\begin{eqnarray*}
K_t:=\sum_{i=0}^{+\infty}K_it^{i},~ \text{ for } K_i\in \mbox{Hom} (V, \mathfrak{g}) ~ \text{ with } K_0 = K.
\end{eqnarray*}
Extending $K_t$ to a linear map from $V [[t]]$ to
$\mathfrak{g}[[t]]$ by $\mathbb{C}[[t]]$-linearity, which is still denoted by $K_t$.

\begin{defn}
A {\bf formal deformation} of $K$ is the formal sum $K_t=\sum_{i=0}^{+\infty}K_it^{i}$ with $K_0 = K$ that satisfying
\begin{align}\label{deform-eqn}
&&K_tu\cdot K_tv=K_t \big( \mathcal{L}_{K_tu}v+\mathcal{R}_{K_tv}u+H(K_t(u), K_t(v)) \big) \tforall u, v \in V.
\end{align}
\end{defn}
Note that $K_t$ is a \relwtrey operator on the pre-Lie algebra $\mathfrak{g}[[t]]$ with respect to the representation $V[[t]]$ and the $2$-cocycle $H$.
If $K_t=\sum_{i=0}^{+\infty}K_it^{i}$ is a formal deformation of a \relwtrey operator $K$ on a pre-Lie algebra $(\mathfrak{g}, \cdot)$  with respect to a representation $(V; \mathcal{L}, \mathcal{R})$ and the $2$-cocycle $H$, then $\cdot_{K_t}$ given by
\begin{eqnarray*}
u \cdot_{K_t} v:=\sum_{i=0}^{+\infty} \left( \mathcal{L}_{K_iu}v+\mathcal{R}_{K_iv}u+ \sum_{j+k = i} H(K_ju, K_k v) \right)t^{i} \tforall u, v\in V,
\end{eqnarray*}
is a formal deformation of the associated pre-Lie algebra $(V,   \cdot_{K})$.

By  (\ref{deform-eqn}) and comparing the coefficients of various powers of $t$, we get
\begin{align*}
 \sum_{i+j=n\geq 0}(K_i u \cdot K_j v )= \sum_{i+j=n\geq 0}K_i(\mathcal{L}_{K_ju}v+\mathcal{R}_{K_jv}u)+\sum_{i+j+k=n\geq 0}K_iH(K_j(u), K_k(v)) \tforall u, v\in V.
\end{align*}
When $n = 0$,  $K$ is a \relwtrey operator. When $n = 1$, we have
\begin{align*}
 Ku \cdot K_1v+K_1u\cdot  Kv =~& K_1(\mathcal{L}_{Ku}v+\mathcal{R}_{Kv}u+H(Ku, Kv))\\
 &+K(\mathcal{L}_{K_1u}v+\mathcal{R}_{K_1v}u+H(K_1(u), Kv)+H(K(u), K_1v)),
\end{align*}
which is equivalent to $(\partial_K (K_1))(u, v) = 0$.

\begin{prop}\label{linear-term-cocycle}
Let $K_t=\sum_{i=0}^{+\infty}K_it^{i}$ be a formal deformation of a \relwtrey operator $K$. Then $K_1$ is a $1$-cocycle in the cohomology of the \relwtrey operator
$K$, that is, $\partial_{K} (K_1)=0$.
\end{prop}

\begin{defn}
The $1$-cocycle $K_1$ is called the {\bf infinitesimal} of the formal deformation $K_t = \sum_{i=0}^{+\infty}K_it^{i}$.
\end{defn}

We proceed to define an equivalence between two formal deformations of a \relwtrey operator.

\begin{defn}
Two formal deformations $K_t =\sum_{i=0}^{+\infty}K_it^{i}$ and $K'_t =\sum_{i=0}^{+\infty}K'_it^{i}$ of a \relwtrey operator $K$ are  {\bf equivalent} if there exist an element $x \in \mathfrak{g}$, linear maps
$\phi_i \in gl(\mathfrak{g})$ and $\psi_i \in gl(V)$, $i \geq 2$ such that the pair
\begin{align*}
\left(\phi_t =\id_\mathfrak{g} + t (L_x-R_x) + \sum_{i=2}^{+\infty}\phi_i t^{i},~ \psi_t=\id_V + t(\mathcal{L}_{x}-\mathcal{R}_{x}+H(x, K-)) + \sum_{i=2}^{+\infty}\psi_i t^{i} \right)
\end{align*}
is a morphism of \relwtrey operators from $K_t$ to $K'_t$.
\end{defn}
By comparing the coefficients of $t$ from
both sides of the identity $\phi_t \circ K_t = K'_t\circ \psi_t$, we obtain
\begin{eqnarray*}
&&K_1(u)-K'_1(u)=K(\mathcal{L}_{x}u-\mathcal{R}_{x}u+H(x, Ku))-x \cdot Ku+Ku \cdot x=\partial_{K}(x)(u) \tforall u \in V.
\end{eqnarray*}

As a consequence, we have

\begin{theorem}
The infinitesimal of a formal deformation of a \relwtrey operator $K$ is a $1$-cocycle in the cohomology of $K$.
The corresponding cohomology class depends only on the equivalence class of the deformation of $K$.
\end{theorem}

\begin{defn}
A \relwtrey operator $K$ is said to be {\bf rigid} if any formal deformation of $K$ is
equivalent to the undeformed deformation $K'_t = K$ .
\end{defn}

At the end of this section, we give a sufficient condition for the rigidity of a \relwtrey operator in terms of Nijenhuis elements.


\begin{defn}
 Let $K$ be a \relwtrey operator. An element $x\in \mathfrak{g}$ is called a {\bf Nijenhuis element} associated with $K$ if $x$ satisfies
\begin{eqnarray*}
x \cdot \overline{\mathcal{R}}_{u}(x)- \overline{\mathcal{R}}_{u}(x)\cdot x=0, ~\text{ for } u \in V
\end{eqnarray*}
and ~(\ref{alg-map-imply})--(\ref{h-comp-imply}) hold.
\end{defn}

We denote the set of Nijenhuis elements associated with $K$  by Nij$(K)$. Note  that a trivial deformation induces a Nijenhuis element.
\begin{theorem}
Let $K$ be a \relwtrey operator. If $Z^1_K (V, \mathfrak{g}) = \partial_K (\mathrm{Nij}(K ))$, then $K$ is rigid.
\end{theorem}
\begin{proof}
Suppose that $K_t=\sum_{i=0}^{+\infty}K_it^{i}$ is a formal deformation of $K$.
By Proposition \ref{linear-term-cocycle},  $K_1$ is a $1$-cocycle in the
cohomology of $K$, that is, $K_1 \in Z^1_K (V, \mathfrak{g})$. Then by the hypothesis, there is a Nijenhuis element $x \in \mathrm{Nij}(K)$
such that $- \partial_K(x)=K_1$. Set
\begin{eqnarray*}
&& \phi_t =\id_\mathfrak{g} + t(L_x-R_x) ~~ \text{ and } ~~\psi_t = \id_V + t(\mathcal{L}_x-\mathcal{R}_x+H(x, K-)),
\end{eqnarray*}
and define $K'_t= \phi_t \circ K_t \circ \psi_t^{-1}$. Then $K'_t$ is a formal deformation equivalent to $K_t$. For $u \in V$, we have
\begin{align*}
 K'_t(u)=~&(\id_\mathfrak{g} + t(L_x-R_x))(K_t(u - t \mathcal{L}_{x}u+\mathcal{R}_{x}u- tH(x,Ku) + \text{ power of } t^{ \geq 2}))\\
=~& K(u) + t (K_1 u - K \mathcal{L}_{x} u+K\mathcal{R}_{x}u - KH (x, Ku) + x\cdot Ku-Ku\cdot x)+ \text{ power of } t^{\geq 2}. \\
=~& K(u) + t^2 K'_2 (u) + \cdots \qquad (\text{as } K_1 = - \partial_K (x)).
\end{align*}
Then the coefficient of $t$ in the expression of $K_t'$ is trivial. Applying the same process repeatedly, we get that $K_t$ is equivalent to $K$ and so
$K$ is rigid.
\end{proof}
\section{ NS-pre-Lie algebras}\label{sec:NS-prelie}

In this section, we introduce the concept of NS-pre-Lie algebras as the underlying structure of \relwtrey operators. We then study some properties of NS-pre-Lie algebras and give some examples.
We show  NS-pre-Lie algebras naturally induce pre-Lie algebras.


\begin{defn}\label{defn-ns}
An {\bf NS-pre-Lie algebra} is a quadruple $(A,  \triangleright, \triangleleft,  \circ )$ consisting of a  vector space $A$ together with three bilinear operations $\triangleright, \triangleleft, \circ: A \otimes A \rightarrow  A$   satisfying, for all $x, y,z \in A$,
\begin{eqnarray*}
&&(A1)\quad  (x \ast y)\triangleright z - x \triangleright (y\triangleright z)= (y \ast x)\triangleright z - y \triangleright (x\triangleright z), \\
&&(A2)\quad x\triangleright (y \triangleleft z ) - (x\triangleright y)\triangleleft z=y\triangleleft (x\ast z) -(y\triangleleft x)\triangleleft z,\\
&&(A3)\quad   (x \ast  y) \circ z - x \circ  (y\ast z)+(x\circ y)\triangleleft z-x\triangleright(y\circ z) \\
&&\quad\quad =  (y \ast  x) \circ z - y \circ  (x\ast z)+(y\circ x)\triangleleft z-y\triangleright(x\circ z),
\end{eqnarray*}
where $x \ast y := x \triangleright y + x\triangleleft y + x \circ y$.
\end{defn}

\begin{remark}
\begin{enumerate}
\item Any $L$-dendriform algebras introduced in \cite{BLN10} is an NS-pre-Lie algebra by taking $\circ$ to be trivial.
\item Any NS-algebra $(A, \succ, \prec, \curlyvee)$ is an NS-pre-Lie algebra by taking
$x\triangleright  y = x  \succ y, x \triangleleft y = x \prec y,  x\circ y=x \curlyvee y$.
\end{enumerate}
\end{remark}

The following result shows that NS-pre-Lie algebras can split pre-Lie  algebras.

\begin{prop}
Let $(A, \triangleright, \triangleleft, \circ )$ be an NS-pre-Lie algebra. Then the vector space $A$ with the bilinear operation
\begin{align*}
\ast : A \otimes A \rightarrow A, ~x\otimes y\mapsto x \ast y:= x \triangleright y + x\triangleleft y + x \circ y
\end{align*}
is a pre-Lie algebra.
\end{prop}
\begin{proof} On one hand, by summing the left hand sides of the identities (A1)-(A3), we have $(x \ast y) \ast z-x \ast (y \ast z)$. On the other hand, by summing the right hand sides of the identities (A1)-(A3), we get $(y \ast x) \ast z- y \ast (x \ast z)$. This completes the proof.
\end{proof}

We call  $(A, \ast)$  the subadjacent pre-Lie algebra of $(A,  \triangleright, \triangleleft, \circ)$ and $(A, \triangleright, \triangleleft, \circ)$ is called a compatible NS-pre-Lie algebra structure on $(A, \ast).$

\begin{prop}\label{prop-ns}
Let $(\mathfrak{g}, \cdot)$ be a pre-Lie algebra and $N: \mathfrak{g} \rightarrow \mathfrak{g}$ be a Nijenhuis operator on it. Define the bilinear operations
\begin{eqnarray*}
x \triangleright y = N(x)\cdot y,~~ x \triangleleft y = x\cdot N(y) ~~ \text{ and } ~~ x \circ y = -N(x\cdot y)~\text{ for }~ x, y\in \mathfrak{g}.
\end{eqnarray*}
Then $(\mathfrak{g}, \triangleright, \triangleleft, \circ)$   is an NS-pre-Lie algebra.
\end{prop}

\begin{proof}  For any $x, y, z\in \mathfrak{g}$, we have
\begin{align*}
    ~&(x \ast y)\triangleright z-x\triangleright(y\triangleright z)\\
    =~& N (x \ast y)\cdot z-x\triangleright (N (y) \cdot  z)\\
    =~& (N(x)\cdot N(y))\cdot z -N(x)\cdot( N (y) \cdot z)\\
    =~& (N(y)\cdot N(x))\cdot z -N(y)\cdot( N (x) \cdot z)\\
    =~&(y \ast x)\triangleright z-y\triangleright(x\triangleright z).
\end{align*}
Then the identity (A1) of Definition \ref{defn-ns} holds. We also have
\begin{align*}
   ~&  y\triangleleft (x\ast z) -(y\triangleleft x)\triangleleft z\\
    =~&y\cdot N(x\ast z)  - (y\cdot N(x))\triangleleft z\\
    =~& y \cdot (N(x)\cdot N(z))-(y\cdot N(x))\cdot N(z) \\
     =~& N(x) \cdot (y\cdot N(z))-(N(x)\cdot y)\cdot N(z)\\
    =~& x\triangleright (y \triangleleft z ) - (x\triangleright y)\triangleleft z.
\end{align*}
Then the identity (A2) holds. To prove the identity (A3), we first recall from \cite{WS19} that the given  pre-Lie algebra structure $\cdot$ and the deformed  pre-Lie algebra structure $\cdot_N$ given in (\ref{deformed-leib}) are compatible in sense that their sum also defines a  pre-Lie algebra structure on $\mathfrak{g}$. This is equivalent to
\begin{align}\label{comp-equiv}
     (x \cdot_N y) \cdot_N z-x \cdot_N (y \cdot_N z) = (y \cdot_N x) \cdot_N z- y \cdot_N (x \cdot_N z),
\end{align}
for $x, y, z \in \mathfrak{g}$. Then the identity (A3) of Definition \ref{defn-ns}  holds by (\ref{comp-equiv}). Thus $(\mathfrak{g}, \triangleright, \triangleleft, \circ)$ is an NS-pre-Lie algebra.
\end{proof}

\begin{exam}
Let $(\mathfrak{g}, \cdot)$ be a 2-dimensional pre-Lie algebra with a basis $\{e_1, e_2\}$ whose non-zero
products are given as follows:
\begin{eqnarray*}
e_2 \cdot e_1 = -e_1,~~~~~~ e_2 \cdot e_2 = e_2.
\end{eqnarray*}
In \cite{WS19}, the authors verified  that all
$N=\left(
\begin{array}{ccc}
 c & d  \\
  0 & c \\
   \end{array}
   \right)$  with $c\neq 0$ are invertible Nijenhuis
operators on the pre-Lie algebra $(\mathfrak{g}, \cdot)$. By Proposition \ref{prop-ns},    $(\mathfrak{g}, \triangleright, \triangleleft, \circ)$ is a 2-dimensional NS-pre-Lie algebra whose nonzero
products are given as follows:
\begin{eqnarray*}
&&e_2 \triangleright e_1 = -ce_1,~~~~~~~~e_2 \triangleright e_2=ce_2,\\
&& e_2 \triangleleft e_1 = -ce_1,~~~~~~~~  e_2 \triangleleft e_2 = -de_1+ce_2,\\
 &&e_2 \circ e_1 = ce_1,~~~~~~~~~~ e_2 \circ e_2 =-de_1-ce_2.
\end{eqnarray*}
\end{exam}

\begin{exam}
Let $(\mathfrak{g}, \cdot)$ be a 3-dimensional pre-Lie algebra with a basis $\{e_1, e_2, e_3\}$ whose nonzero
products are given as follows:
\begin{eqnarray*}
e_3 \cdot e_2 = e_2,~~~~~ e_3 \cdot e_3 = -e_3.
\end{eqnarray*}
In \cite{WS19}, the authors verified  that all
$N=\left(
\begin{array}{ccc}
 d & 0  & 0 \\
  0 & e & f \\
  0 & 0 & e \\
   \end{array}
   \right)$  with $de\neq 0$ are invertible Nijenhuis
operators on the pre-Lie algebra $(\mathfrak{g}, \cdot)$. By Proposition \ref{prop-ns},    $(\mathfrak{g}, \triangleright, \triangleleft, \circ)$ is a 3-dimensional NS-pre-Lie algebra whose nonzero
products are given as follows:
\begin{eqnarray*}
&&e_3 \triangleright e_2 = ee_2,~~~~~~~~e_3 \triangleright e_3=-ee_3,\\
&& e_3 \triangleleft e_2 = ee_2,~~~~~~~~  e_3 \triangleleft e_3 = fe_2-ee_3,\\
 &&e_3 \circ e_2 = -ee_2,~~~~~~ e_3 \circ e_3 =fe_2+ee_3.
\end{eqnarray*}
\end{exam}

\medskip

Let $(A,  \triangleright, \triangleleft, \circ)$ be an NS-pre-Lie algebra. Define two linear maps $\mathcal{L}_{\triangleright}: A \rightarrow gl(A)$,~
$\mathcal{R}_{\triangleleft}: A \rightarrow gl(A)$ and a bilinear map $H : A \otimes A \rightarrow A$ by
\begin{align*}
\mathcal{L}_{\triangleright}(x)y=x\triangleright y,~~~~~~~ \mathcal{R}_{\triangleleft}(y)x=x\triangleleft y,~~~~~~H(x, y)=x \circ y~ \text{ for } x, y\in A.
\end{align*}

\begin{prop}\label{ns-compatible}
Let $(A,  \triangleright, \triangleleft, \circ)$ be an NS-pre-Lie algebra. Then $(A;  \mathcal{L}_{\triangleright}, \mathcal{R}_{\triangleleft})$ is a representation of the subadjacent pre-Lie algebra $(A, \ast)$, and $H$  is a $2$-cocycle. Moreover, the identity map $\id: A \rightarrow A$ is a \relwtrey operator on the pre-Lie algebra $(A, \ast)$ with respect to the representation $(A;  \mathcal{L}_{\triangleright}, \mathcal{R}_{\triangleleft})$.
\end{prop}

\begin{proof}  For any $x, y, z\in A$, we have
\begin{eqnarray*}
&& \mathcal{L}_{\triangleright}(x {\ast} y)z -\mathcal{L}_{\triangleright}(x)\mathcal{L}_{\triangleright}(y)z\\
 &=& (x {\ast} y) \triangleright z- x \triangleright (y\triangleright z)\\
&\stackrel{(A1)}{=}& (y \ast x)\triangleright z - y \triangleright (x\triangleright z)\\
&=& \mathcal{L}_{\triangleright}(y {\ast} x)z -\mathcal{L}_{\triangleright}(y)\mathcal{L}_{\triangleright}(x)z,
\end{eqnarray*}
and
\begin{eqnarray*}
&& \mathcal{L}_{\triangleright}(x)\mathcal{R}_{\triangleleft}(z)y-\mathcal{R}_{\triangleleft}(z)\mathcal{L}_{\triangleright}(x)y\\
&=& x\triangleright(y\triangleleft z)- (x\triangleright y)\triangleleft z\\
&\stackrel{(A2)}{=}& y \triangleleft(x\ast z)-(y\triangleleft x)\triangleleft z\\
&=& \mathcal{R}_{\triangleleft}(x\ast z)y-\mathcal{R}_{\triangleleft}(z)\mathcal{R}_{\triangleleft}(x)y.
\end{eqnarray*}
Then $(A;  \mathcal{L}_{\triangleright}, \mathcal{R}_{\triangleleft})$ is a representation of the subadjacent pre-Lie algebra $(A, \ast)$. Moreover, the condition (A3) is equivalent to $H$ is a $2$-cocycle in the  cochain complex of the pre-Lie algebra $(A, \ast)$ with coefficients in the representation $(A; \mathcal{L}_\triangleright, \mathcal{R}_\triangleleft)$. Finally, we have
\begin{eqnarray*}
\id(\mathcal{L}_{\triangleright}(\id ~x)y+\mathcal{R}_{\triangleleft}(\id ~y)x+H(\id~x, \id ~y))=x\triangleright y+x\triangleleft y+x\circ y=\id ~x {\ast}\id ~y,
\end{eqnarray*}
which implies that $\id : A \rightarrow A$ is a \relwtrey operator on the pre-Lie algebra
$(A, \ast)$ with respect to the representation $(A;  \mathcal{L}_{\triangleright}, \mathcal{R}_{\triangleleft})$.
\end{proof}

\begin{prop}\label{twisted-rota-ns}
Let $(\mathfrak{g}, \cdot)$ be a pre-Lie algebra, $(V; \mathcal{L}, \mathcal{R})$  a representation of $\frakg$ and $H \in C^2 (\mathfrak{g}, V)$ a $2$-cocycle. Let $K : V \rightarrow \mathfrak{g}$ be a \relwtrey operator. Then there exists an NS-pre-Lie algebra on $V$ with bilinear operations defined by
\begin{eqnarray*}
u \triangleleft v :=\mathcal{R}_{Kv}u,~~~~ u \triangleright v := \mathcal{L}_{Ku}v,~~~ u \circ v := H(Ku, Kv) \tforall u, v\in V.
\end{eqnarray*}
\end{prop}

\begin{proof}  For any $u, v, w\in V$, we have
\begin{eqnarray*}
&&(u \ast v)\triangleright w - u \triangleright (v\triangleright w)\\
&=& \mathcal{L}_{K(u\ast v)}w-\mathcal{L}_{Ku}\mathcal{L}_{Kv}w\\
&=& \mathcal{L}_{Ku\ast Kv}w-\mathcal{L}_{Ku}\mathcal{L}_{Kv}w\\
&=&\mathcal{L}_{Kv\ast Ku}w-\mathcal{L}_{Kv}\mathcal{L}_{Ku}w\\
&=& (v \ast u)\triangleright w - v \triangleright (u\triangleright w)
\end{eqnarray*}
and
\begin{eqnarray*}
&& v\triangleleft (u\ast w)-(v\triangleleft u)\triangleleft w\\
&=& \mathcal{R}_{K(u\ast w)}v-\mathcal{R}_{Kw}\mathcal{R}_{Ku}v\\
&=&\mathcal{L}_{Kv}\mathcal{R}_{Kw}u-\mathcal{R}_{Kw}\mathcal{L}_{Kv}u\\
&=& u \triangleright (v \triangleleft w)-(u \triangleright v )\triangleleft w.
\end{eqnarray*}
Hence identities (A1) and (A2) of Definition \ref{defn-ns} holds. Since $H$ is a $2$-cocycle, we have
\begin{align*}
(\partial H)(Ku, Kv, Kw) = 0,
\end{align*}
that is,
\begin{align*}
   \mathcal{L}_{Ku}H(Kv, Kw) -& \mathcal{L}_{Kv}H(Ku, Kw) - \mathcal{R}_{Kw}H(Ku, Kv)+ \mathcal{R}_{Kw}H(Kv, Ku) \\
   -& H(Kv, Ku\cdot Kw) +H(Ku, Kv\cdot Kw)- H([Ku, Kv]_\mathfrak{_g}, Kw) =0.
\end{align*}
This is equivalent to the condition (A3) of Definition \ref{defn-ns}. This completes the proof.
\end{proof}

\begin{remark}
The subadjacent pre-Lie algebra of the NS-pre-Lie algebra constructed in Proposition \ref{twisted-rota-ns} is given by
\begin{align*}
    u \ast v =\mathcal{L}_{Ku} v + \mathcal{R}_{Kv} u + H (Ku, Kv), ~ \text{ for } u, v \in V.
\end{align*}
This pre-Lie algebra on $V$ coincides with the one given in Proposition \ref{induced-leib}.
\end{remark}

\medskip

At the end of this section, we give a necessary and sufficient condition for the existence of a compatible NS-pre-Lie algebra structure on a pre-Lie algebra.

\begin{prop}
Let $(\mathfrak{g}, \cdot)$ be a pre-Lie algebra. Then there is a compatible NS-pre-Lie algebra structure on $\mathfrak{g}$ if and only if there exists an invertible \relwtrey operator $K: V\rightarrow \mathfrak{g}$  with respect to a representation $(V;  \mathcal{L}, \mathcal{R})$ and a $2$-cocycle $H$. Furthermore, the compatible NS-pre-Lie algebra on $\mathfrak{g}$ is given by
\begin{align*}
x\triangleleft y :=K(\mathcal{R}_yK^{-1}x),~~~ x\triangleright y:=K(\mathcal{L}_xK^{-1}y),~~~ x \circ y=KH(x,y)  \tforall  x, y\in \mathfrak{g}.
\end{align*}
\end{prop}
\begin{proof}  Let $K: V\rightarrow \mathfrak{g}$ be an invertible \relwtrey operator on $\mathfrak{g}$ with respect to a representation $(V;  \mathcal{L}, \mathcal{R})$ and a $2$-cocycle $H$. Then it follows from  Proposition \ref{twisted-rota-ns} that there is an NS-pre-Lie algebra on $V$ defined by
\begin{eqnarray*}
u~\Bar{\triangleleft}~ v:=\mathcal{R}_{Kv}u,~~~~~u ~\Bar{\triangleright}~ v:=\mathcal{L}_{Ku}v,~~~~~u ~\Bar{\circ}~ v:=H(Ku, Kv)\, \text{ for all }\, u, v\in V.
\end{eqnarray*}
Since $K$ is an invertible map, the bilinear operations
\begin{eqnarray*}
&&x \triangleleft y := K(K^{-1} y ~\Bar{\triangleleft}~ K^{-1} x) = K(\mathcal{R}_yK^{-1}x), \\
&&x \triangleright y := K(K^{-1} x ~\Bar{\triangleright}~ K^{-1}y) = K(\mathcal{L}_xK^{-1}y),\\
&& x \circ y := K(K^{-1} x ~\Bar{\circ}~ K^{-1}y) =KH(x,y), ~\text{ for } x, y \in \mathfrak{g}
\end{eqnarray*}
defines an NS-pre-Lie algebra on $\mathfrak{g}$. Moreover,  we have
\begin{eqnarray*}
 && x \triangleright y+x \triangleleft y+x\circ y\\
 &=&K(\mathcal{L}_xK^{-1}y)+ K(\mathcal{R}_yK^{-1}x)+KH(x,y) \\
 &=&K(\mathcal{L}_{K\circ K^{-1}x}K^{-1}y)+K(\mathcal{R}_{K\circ K^{-1}y}K^{-1}x)+KH(K\circ K^{-1}x,K\circ K^{-1}y)\\
 &=&(K\circ K^{-1}x) \ast (K\circ K^{-1}y) = x\ast y.
\end{eqnarray*}
Conversely, let $(\mathfrak{g}, \triangleright, \triangleleft, \circ)$ be a compatible NS-pre-Lie algebra structure on $\mathfrak{g}$. Then it follows from Proposition \ref{ns-compatible} that $(\mathfrak{g}; \mathcal{L}_{\triangleright}, \mathcal{R}_{\triangleleft})$ 
is a representation of the pre-Lie algebra $(\mathfrak{g}, \cdot)$, 
and the identity map $\id: \mathfrak{g} \rightarrow \mathfrak{g}$ is a \relwtrey operator on the pre-Lie algebra $(\mathfrak{g}, \cdot)$ with respect to the representation $(\mathfrak{g};   \mathcal{L}_{\triangleright}, \mathcal{R}_{\triangleleft})$. This completes the proof.
\end{proof}

\noindent {\bf Acknowledgments}:
This research is supported by the National Natural Science Foundation of China (Grant No.\@ 12161013, 12101316). We thank professor Li Guo,   Dr. Apurba Das and Dr. Kai Wang for useful discussion.


\begin{thebibliography}{ab}

\bibitem{A11} A. $\acute{\mathrm{A}}$lvarez, C. Sancho, P. Sancho, Reynolds operator on functors, {\em J. Pure Appl. Algebra} {\bf 10} (2011), 1958-1966.

\bibitem{B60} G. Baxter, An analytic problem whose solution follows from a simple algebraic identity, {\em Pacific J. Math.} {\bf 10} (1960), 731-742.

\bibitem{B08} C. Bai, Left-symmetric bialgebras and an analogue of the classical Yang-Baxter equation, {\em Commun. Contemp. Math.}  {\bf 10} (2008), 221-260.

\bibitem{BBGN13} C. Bai, O. Bellier, L. Guo and X. Ni, Splitting of operations, Manin products, and Rota-Baxter operators, {\em Int. Math. Res. Not.} {\bf 3} (2013), 485-524.

\bibitem{BLN10} C. Bai, L. Liu and X. Ni, Some results on $L$-dendriform algebras, {\em J. Geom. Phys.} {\bf 60} (2010), 940-950.

\bibitem{B06} D. Burde, Left-symmetric algebras and pre-Lie algebras in geometry and physics,  {\em Cent. Eur. J. Math.} {\bf 4} (2006), 323-357.

\bibitem{CL01}   F. Chapoton and M. Livernet, Pre-Lie algebras and the rooted trees operad, {\em Int. Math. Res. Not.} {\bf 8} (2001), 395-408.

\bibitem{CHMM} T. Chtioui, A. Hajjaji, S. Mabrouk, A. Makhlouf, Twisted O-operators on 3-Lie algebras and 3-NS-Lie algebras,  arXiv:2107.10890.

\bibitem{Chu05} H. Chu, S.-J. Hu, M.-C. Kang, A variant of the Reynolds operator, {\em Proc. Am. Math. Soc.} {\bf 133}  (2005), 2865-2871.

\bibitem{Das} A. Das, Deformations of associative Rota-Baxter operators, {\em J. Algebra} {\bf 560} (2020), 144-180.

\bibitem{Da20}  A. Das, Twisted Rota-Baxter operators and Reynolds operators on Lie algebras and NS-Lie algebras, {\em J. Math. Phys.} {\bf 62} (2021), 091701.

\bibitem{DG21}  A. Das and S. Guo,  Twisted  relative Rota-Baxter operators on Leibniz algebras and NS-Leibniz algebras, preprint (2020),  arXiv:2102.09752.

\bibitem{DJ57} M.-L. Dubreil-Jacotin, Etude algebrique des transformations de Reynolds, in: Colloque d'alg$\grave{\mathrm{e}}$bre sup$\acute{\mathrm{e}}$rieure, tenu $\grave{\mathrm{a}}$ Bruxelles du 19 au 22 d$\acute{\mathrm{e}}$cembre 1956, 1957, pp. 9-27.

\bibitem{Gub} L. Guo, An Introduction to Rota-Baxter Algebra, Higher Education Press, Beijing, 2012.

\bibitem{GGL} L. Guo, R. Gustavson and Y. Li, Generalized Reynolds algebras and separate Volterra integral equations, in preparation.

\bibitem{HS21}S. Hou and Y. Sheng, Twisted Rota-Baxter operators on 3-Lie algebras and NS-3-Lie algebras, arXiv:2107.13950.

\bibitem{HS212}S. Hou and Y. Sheng, Reynolds $n$-Lie algebras and NS-$n$-Lie algebras, 	arXiv:2108.04407.

\bibitem{KF}J. Kamp$\acute{\mathrm{e}}$ de F$\acute{\mathrm{e}}$riet, Introduction to the Statistical Theory of Turbulence, Correlation and Spectrum, Lecture Series, vol. 8,
The Institute of Fluid Dynamics and Applied Mathematics, University of Maryland, 1951.

\bibitem{LST}A. Lazarev, Y. Sheng and R. Tang, Deformations and homotopy theory of relative Rota-Baxter Lie algebras, {\em Comm. Math. Phys.} {\bf 383} (2021), 595-631.

\bibitem{L04} P. Leroux, Construction of Nijenhuis operators and dendriform trialgebras, {\em Int. J. Math. Math. Sci.} {\bf 49} (2004), 2595-2615.

\bibitem{LHB07}  X. Li, D. Hou and C. Bai, Rota-Baxter operators on pre-Lie algebras,  {\em J. Nonlinear Math. Phys.} {\bf 14} (2007), 269-289.

\bibitem{L20} J. Liu,  Twisting on pre-Lie algebras and quasi-pre-Lie bialgebras,  arXiv:2003.11926.

\bibitem{PBG17} J. Pei, C. Bai and L. Guo, Splitting of operads and Rota-Baxter operators on operads, {\em Appl. Categ. Structures} {\bf 25} (2017), 505-538.

\bibitem{PN15} J. Pei, L. Guo, Averaging algebras, Schr\``oder numbers, rooted trees and operads, {\em J. Algebraic Comb.} {\bf 42} (2015), 73-109.

\bibitem{Rey95} O. Reynolds, On the dynamical theory of incompressible viscous fluids and the determination of the criterion, Philos. {\em Trans. R. Soc. A} {\bf 136} (1895) 123-164; reprinted in Proc. R. Soc. Lond. Ser. A 451 (1941) (1995) 5-47.

\bibitem{Rot64}G.-C. Rota, Reynolds operators, in: Proceedings of Symposia in Applied Mathematics, vol. XVI, Amer. Math. Soc., Providence, R.I., 1964, pp. 70-83.

\bibitem{SW01} P. \u{S}evera and A. Weinstein, Poisson geometry with a 3-form background, {\em Progr. Theoret. Phys. Suppl.} {\bf 144} (2001), 145-154.

\bibitem{S92}  J. Stasheff, Differential graded Lie algebras, quasi-Hopf algebras and higher homotopy algebras. Quantum groups (Leningrad, 1990), 120-137, Lecture Notes in Math., 1510, Springer, Berlin, 1992.

\bibitem{TBGS19} R. Tang, C. Bai, L. Guo and Y. Sheng, Deformations and their controlling cohomologies of $\mathcal{O}$-operators, {\em Comm. Math. Phys.} {\bf 368} (2019), 665-700.

\bibitem{THS21} R. Tang, S. Hou and Y. Sheng, Lie 3-algebras and deformations of relative Rota-Baxter operators on 3-Lie algebras, {\em J. Algebra} {\bf 567} (2021), 37-62.

\bibitem{U08} K. Uchino, Quantum analogy of Poisson geometry, related dendriform algebras and Rota-Baxter operators, {\em Lett. Math. Phys.} {\bf 85} (2008), 91-109.

\bibitem{WS19}  Q. Wang, Y. Sheng, C. Bai and J. Liu, Nijenhuis operators on pre-Lie algebras,  {\em Commun. Contemp. Math.} {\bf 21} (2019), 1850050, 37 pp.

\bibitem{WZ21}K. Wang and G. Zhou, Deformations and homotopy theory of Rota-Baxter algebras of any weight, arXiv:2108.06744.

\bibitem{ZtGG} T. Zhang, X. Gao and L. Guo, Reynolds algebras and their free objects from bracketed words and rooted trees, {\em J. Pure Appl. Algebra} {\bf 225} (2021),  106766, 28 pp.

\end{thebibliography}
\end{document}